\documentclass[11pt,reqno]{amsart}
\usepackage{graphicx} % Required for inserting images

\usepackage[a4paper,textwidth=6.25in,textheight=9.15in]{geometry}
\usepackage{amsmath,amssymb,mathtools}
\usepackage{bm}
\usepackage{booktabs,tabularx,array}
\usepackage{enumitem}
\usepackage{microtype}
\usepackage{xcolor}
\usepackage[colorlinks=true,citecolor=green!45!black,linkcolor=blue!55!black,urlcolor=blue!55!black]{hyperref}
\usepackage[normalem]{ulem}
\usepackage{amsmath}
\usepackage{mathrsfs}
\usepackage{soul}
\usepackage{bbm}

\numberwithin{equation}{section}
\setlist[enumerate]{leftmargin=*,topsep=3pt,itemsep=2pt}
\renewcommand{\arraystretch}{1.12}
\newtheorem{theorem}{Theorem}[section]

\newtheorem{proposition}[theorem]{Proposition}
\newtheorem{corollary}[theorem]{Corollary}
\newtheorem{lemma}[theorem]{Lemma}
\theoremstyle{definition}
\newtheorem{definition}[theorem]{Definition}
\newtheorem{problem}[theorem]{Problem}
\theoremstyle{remark}
\newtheorem{remark}[theorem]{Remark}
\newtheorem{example}[theorem]{Example}

\renewcommand{\emph}[1]{{\it#1}}

\newcommand{\one}{\mathbbm{1}}

\newcommand{\T}{\mathbb T}
\newcommand{\N}{\mathbb N}
\newcommand{\R}{\mathbb R}
\newcommand{\Z}{\mathbb Z}
\newcommand{\Q}{\mathbb Q}
\newcommand{\cA}{\mathcal A}

\newcommand{\cH}{\mathcal H}

\newcommand{\cY}{\mathcal Y}
\newcommand{\cM}{\mathcal M}
\newcommand{\cL}{\mathcal L}
\newcommand{\bfd}{\mathbf{d}}
\newcommand{\bfx}{\mathbf{x}}
\newcommand{\orb}{\mathrm{orb}}
\newcommand{\ND}{\mathrm{ND}}
\newcommand{\bfy}{\mathbf{y}}
\newcommand{\bfz}{\mathbf{z}}

\newcommand{\bfp}{\mathbf{p}}

\newcommand{\cG}{\mathcal G}

\newcommand{\Bad}{\mathrm{BAD}}
\newcommand{\BAD}{\mathrm{BAD}}
\newcommand{\tN}{\mathrm{N}}
\newcommand{\td}{\mathrm{d}}
\newcommand{\diam}{\operatorname{diam}}
\newcommand{\supp}{\operatorname{supp}}
\newcommand{\dist}{\operatorname{dist}}
\newcommand{\dimH}{\dim_{\mathrm H}}
\newcommand{\dimD}{\dim_{\mathrm D}}
\newcommand{\dimL}{\dim_{\mathrm L}}

\title[Shrinking targets versus recurrence ]
{Shrinking targets versus recurrence:\\
A survey on the Hausdorff dimension theory}
\author{Zhang-nan Hu}\address{Z.N.~Hu, College of Science,
  China University of Petroleum, Beijing 102249, P. R. China}
\email{hnlgdxhzn@163.com}

\author{Jun-jie Huang} \address{J. J. Huang:  School of Mathematics, South China University of Technology, Guangzhou 510641, P. R. China}
\email{h1135778478@163.com}

\author{Bing Li*}\thanks{* Corresponding author}  \address{B. Li, School of Mathematics, South China University of Technology, Guangzhou 510641, P. R. China}
\email{scbingli@scut.edu.cn}

\date{\today}

\dedicatory{\emph{Dedicated to Sanju Velani on the occasion of his 60th birthday}}

\begin{document}

\begin{abstract}
This survey reviews results on the Hausdorff dimension of shrinking target sets and recurrence sets in dynamical systems. In their basic formulations, shrinking target sets concern orbits approaching a prescribed centre, whereas recurrence sets concern orbits returning  close to their initial points.  A central theme is the comparison of the  Hausdorff dimensions of these two classes of limsup sets. The dynamical systems considered include expanding conformal systems, irrational rotations, affine iterated function systems, and matrix transformations. Related results on Hausdorff measure are  also included when they provide natural refinements of the dimension theory. Finally, we review results on badly approximating orbits and non-recurrent sets, with particular attention to winning properties and intersections with self-conformal and self-affine sets.
\end{abstract}

\maketitle

\section{Introduction: background and motivation}

% Let $(X,\bfd)$ be a separable metric space,% ({\color{red}Huang: I choose to use $\bfd$ (instead of $d$) to denote the metric, because I want to distinguish it from that of $\R^d$ or $\T^d$}), 
% let $\mathcal B$ be a Borel $\sigma$-algebra in $X$,  let $T:X\to X$ be a Borel measurable map, and let $\mu$ be a $T$-invariant Borel probability measure on $X$.

 Let $(X,\bfd)$ be a separable metric space, $\mathcal B$ its Borel $\sigma$-algebra, $T:X\to X$ a Borel measurable map, and $\mu$ a $T$-invariant Borel probability measure on $X$.   The Poincar\'e recurrence theorem implies that
\begin{equation}\label{eq:intro-qualitative-recurrence}
  \liminf_{n\to\infty}\bfd(T^n x,x)=0 \quad\text{for }\mu\text{-almost every }x\in X,
\end{equation}
If $T$ is ergodic, then Birkhoff's ergodic theorem gives the corresponding
fixed target statement: for any $y\in\operatorname{supp}\mu$,
\begin{equation}\label{eq:intro-orbit-approaches-target}
  \liminf_{n\to\infty}\bfd(T^n x,y)=0\quad\text{for }\mu\text{-almost every }x\in X.
\end{equation}
The above \eqref{eq:intro-qualitative-recurrence} and \eqref{eq:intro-orbit-approaches-target} describe two basic forms of qualitative approximation:  a $\mu$-typical orbit may return to its initial point, or  approach a fixed target along a subsequence of times $n\in\N$.  Neither statement, however, records the rate at which the approximation occurs.

To pass from qualitative to quantitative questions, let $\psi:\mathbb N\to(0,\infty)$ satisfy $\psi(n)\to0$.  Given $y\in X$, define the shrinking target set
\begin{equation*}\label{eq:intro-shrinking-target}
  W(T,y,\psi):=\bigl\{x\in X:\bfd(T^n x,y)<\psi(n) \text{ for infinitely many }n\in\mathbb N\bigr\},
\end{equation*}
and define the  quantitatively recurrent set
\begin{equation}\label{eq:intro-recurrence}
  R(T,\psi):=\bigl\{x\in X:\bfd(T^n x,x)<\psi(n) \text{ for infinitely many }n\in\mathbb N\bigr\}.
\end{equation}
These two sets are the basic objects of study in this survey.  Although the two definitions differ only in the centre of the target, this distinction changes the geometry of the level sets. For shrinking target sets, the $n$-th level is the inverse image $T^{-n}B(y,\psi(n))$ of a prescribed ball. For  recurrent sets, the centre depends on the initial point of the orbit itself, and the $n$-th level consists of points whose $n$-th iterates return close to their original point.

Both mentioned sets are dynamically defined $\limsup$ sets.  For any $n\in\N$, let
\[
  W_n(T,y,\psi):=T^{-n}B(y,\psi(n)) \quad\text{and}\quad R_n(T,\psi):=\{x\in X:\bfd(T^n x,x)<\psi(n)\},
\]
then the shrinking target set and recurrent set have brief descriptions as the following
\[  W(T,y,\psi)=\limsup_{n\to\infty}W_n(T,y,\psi),\quad R(T,\psi)=\limsup_{n\to\infty}R_n(T,\psi).\]
Since $\mu$ is $T$-invariant, the $\mu$-measure of $W_n(T,y,\psi)$ is exactly the $\mu$-measure of the ball $B(y,\psi(n))$. This together with the convergent Borel--Cantelli lemma immediately implies that
\begin{equation*}\label{eq:measure-convergence}
 \begin{aligned}
  \sum_{n=1}^{\infty}\mu\bigl(B(y,\psi(n))\bigr)<\infty
  &\quad\Longrightarrow\quad
  \mu\bigl(W(T,y,\psi)\bigr)=0,\\
  \sum_{n=1}^{\infty}\mu\bigl(R_n(T,\psi)\bigr)<\infty
  &\quad\Longrightarrow\quad
  \mu\bigl(R(T,\psi)\bigr)=0.
 \end{aligned}
\end{equation*}
The divergent case is considerably more delicate and typically requires quantitative mixing or quasi-independence estimates. Under suitable assumptions, divergence of the corresponding sum of measures often yields a full-measure conclusion, commonly referred to as dynamical Borel--Cantelli lemma. Stronger versions also provide asymptotic estimates for the number of visits, as follows:
\begin{equation*}
    \sum_{n=1}^N\one_{W_n(T,y,\psi)}(x),\qquad\sum_{n=1}^N\one_{R_n(T,\psi)}(x).
\end{equation*}
For the recurrence problem, an additional difficulty is that invariance alone does not give direct estimate for $\mu(R_n(T,\psi))$. We will discuss prior work in this direction later in this section.

When a shrinking target set or recurrent set is a $\mu$-null set, it is natural to investigate its finer size in terms of Hausdorff dimension and Hausdorff measures. 
Our principal focus is the Hausdorff dimension of shrinking target sets and {recurrent } sets. In several settings, the dimension results are complemented by zero--full laws for Hausdorff measures associated with general dimension functions. The relevant definitions and notation are recalled at the beginning of  {Section~\ref{sec:preliminaries}}. The arguments surveyed below involve a range of covering, measure theoretic, thermodynamic and geometric methods. The general principles  for Hausdorff dimensional estimates used widely in the present literature are reviewed in Section~\ref{sec:preliminaries}.

Beyond the basic sets $W(T,y,\psi)$ and $R(T,\psi)$, the literature considers several natural variants.  { In shrinking target problems, the centre of the target may vary with time. Other extensions concern the geometry of the targets, which may be a symbolic cylinder, a rectangle, or a more general anisotropic set.} The shrinking or recurrent rate may also be determined by an ergodic sum rather than by a prescribed numerical sequence. We include such variants when they retain the underlying approximation structure of shrinking targets or recurrence. After surveying shrinking targets and recurrence separately, we discuss two frameworks in which the two problems admit a common formulation: recurrence under Lipschitz twists and limsup sets generated by matrix sequences. The first applies to a broad class of dynamical systems, whereas the second exploits the linear structure of matrix transformations.

A complementary aspect of the theory concerns badly approximating and non-recurrent orbits. Define the badly approximating set
\begin{equation*}\label{eq:intro-badly-approximating}
  \operatorname{BAD}(T,y):=\bigl\{x\in X:\liminf_{n\to\infty}\bfd(T^n x,y)>0\bigr\}
\end{equation*}
and the non-recurrent set
\begin{equation*}\label{eq:intro-non-recurrent}
  \tN(T):=\bigl\{x\in X:\liminf_{n\to\infty}\bfd(T^n x,x)>0\bigr\}.
\end{equation*}
Under the hypotheses of \eqref{eq:intro-qualitative-recurrence} and \eqref{eq:intro-orbit-approaches-target}, these exceptional sets are $\mu$-null, but in many expanding systems they have full Hausdorff dimension.  Frequently they are much larger in the sense of Schmidt's
game: winning properties are stable under countable intersections and force large intersections with broad classes of fractals.  The final part of the survey develops this complementary theory for matrix sequences and piecewise expanding maps, and discusses intersections with self-conformal and self-affine sets.

The shrinking  target problem has roots in metric Diophantine approximation. For irrational rotations, early measure-theoretic results are already present in the work of Cassels \cite{Cassels1950} and Kurzweil \cite{Kurzweil1955}.  The dimension theory for expanding conformal dynamics was initiated by Hill and Velani \cite{HillVelani1995}, who studied expanding rational maps on their Julia sets,  {with further developments} in \cite{HillVelani1997}.    Subsequent { works} treated expanding Markov maps, continued fractions, $\beta$-transformations, conformal iterated function systems, and toral endomorphisms; see, among others,  
\cite{AllenBarany2021,BugeaudWang2014, HillVelani1999, LiLiaoVelaniZorin2023, LiWangWuXu2014}{, as well as non-conformal self-affine systems \cite{BakerKoivusalo2024, BaranyRams2018}} . Nonautonomous analogues associated with Cantor series expansions were studied in \cite{SunCao2017}.  The mass transference principle \cite{BeresnevichVelani2006} provides a mechanism for deriving Hausdorff-measure statements from full-measure results for related limsup sets.

Quantitative recurrence has a separate history.  A foundational theorem of Boshernitzan \cite{Boshernitzan1993} states that if the $\alpha$-dimensional Hausdorff measure $\mathcal{H}^{\alpha}$ is $\sigma$-finite on $X$ for some $\alpha>0$, then
\[
  \liminf_{n\to\infty}n^{1/\alpha}\bfd(T^n x,x)<\infty \quad\text{for }\mu\text{-almost every }x\in X;
\]
if $\mathcal H^\alpha(X)=0$, the $\liminf$ term is zero.  Barreira and Saussol \cite{BarreiraSaussol2001} subsequently related quantitatively recurrent
rates to the local pointwise dimension of the invariant measure.  General approximation functions led to dynamical Borel--Cantelli lemma in expanding, self-similar, and self-conformal settings; see, among others, \cite{BakerFarmer2021, ChangWuWu2019, HeLiao2026, HussainLiSimmonsWang2022, KleinbockZheng2023}.  Feng and Wu \cite{FengWu2001} studied the Hausdorff dimension of recurrence-rate level sets in symbolic spaces. For $\limsup$ recurrent sets of the form \eqref{eq:intro-recurrence}, Tan and Wang \cite{TanWang2011} obtained completely dimensional results for the $\beta$-transformation. Subsequent work treated conformal systems, homogeneous self-similar sets, affine systems and matrix transformations; see, for example,  \cite{BakerFarmer2021, BakerKoivusalo2024, ChangWuWu2019, HuPersson2024, SeuretWang2015}  and the references therein.  For a broader historical perspective, we refer the reader to the survey  of Wang and Wu \cite{WangWu2017Survey} for an account of the development of Hausdorff dimensional results for shrinking target sets and {recurrent} sets, and to He, Li and Velani \cite{HELIvelani} for an overview of the corresponding measure-theoretic results.

The study of sizes of $\Bad(T,y)$ and $\tN(T)$ dates back to Erd\"{o}s \cite{MR412138} who asked whether there exists an irrational number $\alpha$ such that $\{a_n\alpha\}_{n\in\N}$ is not dense modulo one when $\{a_n\}_{n\in\N}$ is a lacunary sequence of  integers (see Section \ref{badsec} for definition). This problem was completely resolved by Pollington \cite{MR540398}, who showed that if the sequence of real numbers $\{a_{n}\}_{n\in\N}$ is  lacunary, then one has
\begin{equation*}
    \dimH\big(\{x\in[0,1):\inf\{a_nx\ (\mathrm{mod}\ 1):n\in\N\}>0\}\big)\ = \ 1.
\end{equation*}
To see the connection with the dynamical setting, let $m\geq 2$ be a positive integer. By taking $a_n=m^n$ or $a_n=m^n-1$, the above result implies that $\Bad(T_m,0)$ and $\tN(T_m)$ has full Hausdorff dimension, where $T_m$ is the $\times m$ map on $[0,1)$ defined as $T_m(x):=mx\ (\mathrm{mod}\ 1)$. Since then, substantial progress has been made  in the study of sizes of $\Bad(T,y)$ and $\tN(T)$ for various systems; see, for example, \cite{MR2818688,MR980795,MR1440064,hu2025,huang2025win,MR540398,nayuan2021} and references therein.

The core of the survey is organized around four families of dynamical systems for which a substantial Hausdorff theory of shrinking targets and recurrence has been developed. The first consists of expanding conformal systems.  The second consists of irrational rotations, where the shrinking target problem is equivalent to inhomogeneous Diophantine approximation.  Notice that  self-recurrence for a rotation is degenerate in the present formulation, since
$ \|R_\alpha^n x-x\|=\|n\alpha\|$
is independent of $x$; this explains why rotations occur primarily on the shrinking target problem section of the survey.  The third family consists of genuinely non-conformal affine systems. The fourth consists of matrix transformations and, more generally, matrix sequences.
  These four families  are not disjoint, and also do not exhaust the literature.  Related results are also known in several other settings, but we discuss them only briefly
when they illustrate a different method or phenomenon.

The structure of this survey is as follows. Section~\ref{sec:shrinking} surveys shrinking target problem results for expanding conformal systems, irrational rotations, affine iterated function systems, and matrix transformations.  Section~\ref{sec:recurrence} turns to
quantitative recurrence for expanding Markov maps, conformal and affine
iterated function systems, and  matrix transformations.  Section~\ref{sec:comparison} presents two frameworks that place shrinking target sets and {recurrent} sets under a common roof:
recurrence under Lipschitz twists and limsup sets generated by matrix sequences.  The final section studies badly approximating and non-recurrent orbits, first for matrix sequences and piecewise expanding maps and then through Schmidt-type games and intersections with self-conformal and self-affine fractals.

\section{Preliminaries}\label{sec:preliminaries}

This section fixes the notation to be used throughout the survey and records several general principles for estimating Hausdorff measure and dimension. We include only the formulations that will be needed later; assumptions and constructions specific to individual dynamical systems will be introduced in the corresponding sections.

Throughout, for a positive function $\psi:\mathbb N\to(0,\infty)$, we write
\[
\alpha(\psi):=\liminf_{n\to\infty}\frac{-\log\psi(n)}{n}
\]
and call $\alpha(\psi)$ the  \emph{lower exponential order} of $\psi$ at infinity. {In parts of the literature, this quantity is referred to simply as the lower order of $\psi$ at infinity.} When $\psi$ is clear from the context, we simply write $\alpha$.

\subsection{Hausdorff measure and Hausdorff dimension }\label{sec:hausdorff-measures-dimension}

We begin by fixing the definitions and notation for Hausdorff measures and Hausdorff dimension. Let $(X,\bfd)$ be a metric space. For a non-empty set $U\subset X$, write
\[
 \operatorname{diam}U:=\sup\{\bfd(x,y):x,y\in U\}.
\]

A \emph{dimension function} is a continuous non-decreasing function $f:[0,\infty)\to [0,\infty)$ with $f(0)=0$. For $E\subset X$ and $\delta>0$, define 
\[
 \mathcal H^f_\delta(E):=\inf\left\{\sum_{i=1}^{\infty}f(\operatorname{diam}U_i):E\subset\bigcup_{i=1}^{\infty}U_i,\ \operatorname{diam}U_i\leq\delta\right\}.
\]
The $f$-Hausdorff measure of $E$ is
\[
 \mathcal H^f(E):=\lim_{\delta\to0}\mathcal H^f_\delta(E).
\]

For  $f(r)=r^s$ ($s\geq0$), we use the standard notation
\[
 \mathcal H^s:=\mathcal H^f, \quad  \mathcal H^s_\delta:=\mathcal H^f_\delta.
\]
The Hausdorff dimension of $E$ is defined by
\[
 \dim_{\mathrm H}E:=\inf\{s\geq0:\mathcal H^s(E)=0\}=\sup\{s\geq0:\mathcal H^s(E)=\infty\}.
\]
With the above notation, for any subset $E\subset X$, one has $\cH^{0}(E)=\#E$ and 
\[
 \mathcal H^s(E)=\begin{cases}
  \infty, & 0< s<\dim_{\mathrm H}E,\\
  0, & s>\dim_{\mathrm H}E.
 \end{cases}
\]
At the critical exponent $s=\dim_{\mathrm H}E$, however, the Hausdorff measure may be zero, positive and finite, or infinite. This is one reason why a Hausdorff measure statement provides finer information than a dimension formula alone.

For some specific $\limsup$ set $E\subset X$, such as shrinking target set or recurrent set mainly discussed in this paper, a typical zero--full law with respect to $\mathcal H^f$ takes the form
\[
 \mathcal H^f(E)= \begin{cases}
  0, & \text{if a specified series converges},\\
  \mathcal H^f(X), & \text{if the same series diverges}.
 \end{cases}
\]
Some results in the survey give the stronger local conclusion
\[
 \mathcal H^f(E\cap B) = \mathcal H^f(X\cap B)
\]
for every ball $B$ that intersects $X$. We shall state explicitly whether ``full measure'' refers to the global equality or to this local form.

\subsection{General principles for Hausdorff measure and dimension estimates}
\label{subsec:dimension-principles}

Most upper bounds considered in this survey are obtained by constructing suitable covers of the relevant $\limsup$ sets. Lower bounds are more varied and often combine several methods, including mass distribution principle and Cantor-type arguments, energy estimates, ubiquity and mass transference principle, large intersection techniques, and constructions based on Gibbs or equilibrium measures. These methods are not mutually exclusive: Cantor or thermodynamic constructions often provide measures to which local mass or energy estimates are applied, while energy, ubiquity and mass transference principle may in turn be used to establish large intersection properties. The appropriate combination depends on the geometry, overlap, and distribution of the level sets $W_n(T,y,\psi)$ and $R_n(T,\psi)$. 

\subsubsection{Covering estimates}\label{subsubsec:covering-estimates}

The following elementary consequence of the definition of Hausdorff measure is commonly referred to in metric Diophantine approximation as
the Hausdorff--Cantelli lemma; see \cite[Lemma~3.10]{BernikDodson1999}.

\begin{lemma}[Hausdorff--Cantelli lemma] \label{lem:hausdorff-cantelli}
Let
\[
   E=\limsup_{n\to\infty} E_n=\bigcap_{N=1}^{\infty}\bigcup_{n\geq N}E_n,
\]
and suppose that, for each $n$, the set $E_n$ admits a countable cover
\[
   E_n\subseteq\bigcup_{j}U_{n,j}
\]
such that  $ \sup_j\{\operatorname{diam}U_{n,j}\}\to 0\,\,\text{as }n\to\infty.$ If $f$ is a dimension function and
\[
   \sum_{n=1}^{\infty}\sum_jf\bigl(\operatorname{diam}U_{n,j}\bigr)<\infty,
\]
then
\[
   \mathcal H^f(E)=0.
\]
In particular, if
\[
   \sum_{n=1}^{\infty}\sum_j\bigl(\operatorname{diam}U_{n,j}\bigr)^s<\infty,
\]
then \(\dim_{\mathrm H}E\leq s\).
\end{lemma}

For any $n\in\N$, let $\mathcal{C}_n$ be a cover of $E_n$ and write
\[
  Z_n(s):=\sum_{C\in\mathcal C_n} \bigl(\operatorname{diam}C\bigr)^s,
\]
then the Hausdorff--Cantelli lemma gives
\[
   \dim_{\mathrm H}\left(\limsup_{n\to\infty}\bigcup_{C\in\mathcal C_n}C\right)\leq\inf\left\{ s>0:\sum_{n=1}^{\infty}Z_n(s)<\infty \right\}.
\]
In the symbolic systems considered below, the sets in $\mathcal{C}_n$ are indexed by admissible words of length $n$, and their diameters are  connected with additive or subadditive weights along these words. The corresponding covering sums $Z_n(s)$ are therefore closely related to the
partition sums defining the relevant pressure. %In symbolic dynamical systems, the covering sums \(Z_n(s)\) are often the partition sums appearing in the definition of pressure; see Section~\ref{subsec:symbolic-pressure}.

The strategy for choosing the cover $\mathcal{C}_n$ matters greatly. If $E_n$ is the union of some ``regular''-shaped subsets of $X$, such as balls, then an optimal covering of $E_n$  often consists of  balls (or cylinders when $E_n$ is dynamically defined), each of diameters comparable to that of the corresponding subsets. Nevertheless, if $E_n$ is the union of some ``irregular'' subsets, such as rectangles or ellipsoids with strongly direction-dependent scales, then such a cover is too coarse and may result in a loss of essential geometric information. A more suitable approach should account for both the different intrinsic scales and the spacing distribution of the underlying subsets.

\subsubsection{Mass distribution and energy estimates}\label{subsec:measure-lower-bounds}

The most commonly used measure-based criteria for Hausdorff-dimension lower bounds is the mass distribution principle. For a precise statement of the Mass Distribution Principle, see \cite[Section 2]{BeresnevichVelani2006}; for further general information, we refer the reader to \cite{Falconer2003,Mattila1995}.

\begin{proposition}[Mass distribution principle] \label{prop:mass-distribution}
Let $F\subseteq X$, and let $\nu$ be a finite Borel measure supported on $F$ with $\nu(F)>0$. Suppose that there exist constants $C>0$
and $r_0>0$ such that
\[
   \nu(B(x,r))\leq C f(r)
\]
for each $x\in F$ and $0<r<r_0$. Then  $\mathcal H^f(F)\geq\frac{\nu(F)}{C}.$ In particular, if $f(r)=r^s$, 
%\[  \nu(B(x,r))\leq Cr^s,\]
then $ \dim_{\mathrm H}F\geq s.$
\end{proposition}

A useful local formulation is given by the lower local dimension
\[
   \underline d_\nu(x):=\liminf_{r\to0}\frac{\log\nu(B(x,r))}{\log r}.
\]
If $\nu(F)>0$ and
\[
   \underline d_\nu(x)\geq s\quad\text{for $\nu$-almost every $x\in F$,}
\]
then
\[
   \dim_{\mathrm H}F\geq s.
\]

In applications, the required measure is often constructed on a Cantor-type subset of the limsup set. Mass is assigned to the surviving components either uniformly or according to geometric or symbolic weights chosen for the construction, and the lower bound follows once the mass of arbitrary balls has been controlled. Separation, bounded distortion and quantitative distribution estimates are commonly used for this purpose.

A different measure-theoretic criterion is provided by energy estimates. For a finite Borel measure $\nu$ and $s>0$, define its $s$-energy by %({\color{blue}Huang: I would like to use $\td\nu(x)$ instead of $d\nu(x)$})
\[
   I_s(\nu):=\iint \bfd(x,y)^{-s}\,\td\nu(x)\,\td\nu(y).
\]
If $\nu$ is a non-zero finite measure supported on $F$ and satisfies
\[ I_s(\nu)<\infty,\]
then $\mathcal{H}^s(F)=\infty$ and, in particular, $\dim_{\mathrm H}F\geq s.$ See, for example, \cite[Theorem~4.13]{Falconer2003} and \cite[Theorem~4.27]{MortersPeres2010}.

 In applications to $\limsup$ sets, energy estimates may be applied either to a measure supported on the set itself or to a sequence of measures $\nu_n$ supported on approximating sets $E_n$. In the latter case, weak convergence of $\nu_n$ to a suitable  measure, together with uniform energy estimates, can yield a lower bound on Hausdorff dimension, or more strongly a large intersection property, for $\limsup_n E_n$; see \cite{Persson2019, PerssonReeve2015}.

\subsubsection{Mass transference principle}
\label{subsubsec:mass-transference}

In many problems it is easier to prove a full-measure statement for a $\limsup$ set formed by enlarged targets than to construct a measure
directly on the original $\limsup$ set. The mass transference principle of Beresnevich and Velani \cite{BeresnevichVelani2006} converts such a
statement into a Hausdorff-measure statement for the original targets.

The following form of the mass transference principle (MTP), first proved in \cite{BeresnevichVelani2006}, is stated for Ahlfors regular measures and a general dimension function. In \cite{hussain2019}, Hussain and Simmons gave a new proof by developing a broader principle for Hausdorff measure.  We refer to \cite{MR4902008} for the corresponding MTP for Borel measures without the regularity assumption. To state the result,  given a metric space $X$ and $\delta>0$, we say that a Borel measure $m$ on $X$ is  $\delta$-Ahlfors regular measure if there exist constants $c_1,c_2>0$ and $r_0>0$ such that
\[
   c_1r^\delta\leq m(B(x,r)) \leq c_2r^\delta
\]
for each $x\in X$ and $0<r<r_0$. For a ball $B=B(x,r)$ in $X$, we write  $B^f:=B\bigl(x,f(r)^{1/\delta}\bigr).$

\begin{theorem}[Mass transference principle] \label{thm:mass-transference}
Let $X$ be a locally compact metric space carrying a $\delta$-Ahlfors regular measure $m$. Let $ B_n=B(x_n,r_n)$ with $r_n\to0$ as $n\to\infty$.  Let $f$ be a dimension function such that $x^{-\delta}f(x)$ is monotonic. Suppose that, for any ball $B\subseteq X$,
\[
   m\left(B\cap\limsup_{n\to\infty}B_n^{f}\right) = m(B).
\]
Then, for any ball $B\subseteq X$,
\[
   \mathcal H^f\left(B\cap\limsup_{n\to\infty}B_n \right) = \mathcal H^f(B).
\]
\end{theorem}

The theorem above concerns balls. It does not apply to rectangles or other anisotropic targets merely by replacing them with balls of comparable diameter. %Rectangle-to-rectangle mass transference principles require additional geometric hypotheses and will be introduced in the non-conformal settings where they are used; see, for example, \cite{HuHuangLiWu2026,LiLiaoVelaniWangZorin2025,WangWu2021}. {\color{red}Need to add some comments here.}
Extensions to neighbourhoods of more general
resonant sets and to systems of linear forms were established in
\cite{AllenBaker2019,AllenBeresnevich2018}, while mass transference
principles for rectangular targets were developed in
\cite{LiLiaoVelaniWangZorin2025,WangWu2021,WangWuXu2015}. For surveys of
further developments and applications, see
\cite{AllenDaviaud2025,AllenTroscheit2019} and the references therein.

  \subsection{Symbolic notation and pressure}\label{subsec:symbolic-pressure}

Let $\mathcal A$ be a finite or countable alphabet. We write
\[
   \mathcal A^n=\{\mathbf i=i_1\cdots i_n:i_j\in\mathcal A\ \text{for any $j=1,2,...,n$}\},\quad \mathcal A^*=\bigcup_{n\geq1}\mathcal A^n.
\]
For $\mathbf i\in\mathcal A^*$, its length is denoted by $|\mathbf i|$. Give an infinite word $\omega=(\omega_k)_{k\geq1}\in\mathcal A^{\mathbb N}$, we denote  $ \omega|_n=\omega_1\cdots\omega_n$ as its prefix of length $n$. The left shift $\sigma:\mathcal A^{\mathbb N}\to\mathcal A^{\mathbb N}$ is defined by
\[
   \sigma(\omega_1\omega_2\omega_3\cdots)= \omega_2\omega_3\cdots.
\]
We say that a closed subset $\Sigma\subseteq\mathcal{A}^{\N}$ is a subshift if $\sigma(\Sigma)\subseteq\Sigma$. Its language of length $n$ is defined as $\mathcal L_n(\Sigma):= \{\omega|_n:\omega\in\Sigma\}.$ For $\mathbf i\in\mathcal L_n(\Sigma)$, the corresponding cylinder is
\[
   [\mathbf i]:=\{\omega\in\Sigma:\omega|_n=\mathbf i\}.
\]

We first recall the additive classic topological pressure for continuous map of a compact metric space and a continuous potential. The pressures defined over cylinders for countable Markov systems and the subadditive  pressures arising in affine systems will be introduced in the relevant subsections. Let $T:X\to X$ be a map and let $\varphi:X\to\mathbb R$. The $n$-th ergodic sum of $\varphi$ is
\[
   S_n\varphi(x):=\sum_{k=0}^{n-1}\varphi(T^kx).
\]

Suppose now that $X$ is a compact metric space, $T:X\to X$ is continuous and $\varphi\in C(X)$. 
A set $E\subseteq X$ is $(n,\varepsilon)$-separated if 
\[\max_{0\leq k<n}d(T^kx,T^ky)>\varepsilon\]
for all distinct $x,y\in E$. Let
\[
   Z_n(T,\varphi,\varepsilon):=\sup_E \sum_{x\in E}\exp\bigl(S_n\varphi(x)\bigr),
\]
where the supremum is taken over all finite $(n,\varepsilon)$-separated subsets of $X$. The topological pressure of $\varphi$ is
\begin{equation}\label{classic-pressure}
   P(T,\varphi):=\lim_{\varepsilon\to 0}\limsup_{n\to\infty}\frac1n\log Z_n(T,\varphi,\varepsilon).
\end{equation}
By the variational principle \cite[Theorem~9.10]{Walters1982},%[{\color{blue}reference ?}],
\[
   P(T,\varphi)=\sup_{\mu\in\mathcal M_T(X)}\left\{h_\mu(T)+\int_X\varphi\,d\mu\right\},
\]
where $\mathcal M_T(X)$ denotes the set of $T$-invariant Borel probability measures on $X$ {and $h_{\mu}(T)$ denotes the Kolmogorov-Sinai entropy of $T$ with respect to $\mu$}. A measure attaining this supremum is called an equilibrium state, or equilibrium measure, for $\varphi$.

For a subshift $\Sigma$ over a finite alphabet and $\varphi\in C(\Sigma)$, pressure can equivalently be computed from partition sums over admissible words. Let $\mathcal A$ be a finite alphabet. % let $\Sigma\subseteq\mathcal A^{\mathbb N}$ be a subshift, and let $\mathcal L_n$ denote the set of admissible words of length $n$. 
Then
\[
   Z_n^{\mathrm{cyl}}(\varphi):=\sum_{\mathbf i\in\mathcal L_n(\Sigma)}\exp\left(\sup_{\omega\in[\mathbf i]}S_n\varphi(\omega)\right).
\]
These partition sums satisfy the following sub-multiplicative property
\[
   Z_{n+m}^{\mathrm{cyl}}(\varphi)\leq
   Z_n^{\mathrm{cyl}}(\varphi)\cdot
   Z_m^{\mathrm{cyl}}(\varphi).
\]
Consequently,
\[
   P(\sigma,\varphi) =\lim_{n\to\infty}\frac1n\log Z_n^{\mathrm{cyl}}(\varphi).
\]
If the alphabet is countably infinite, the symbolic space need not be compact and the partition sums may be infinite.
The pressures used for countable expanding Markov maps will therefore be defined from the relevant cylinder partition sums where they are
first needed. Subadditive singular value pressures and target dependent modified pressures will likewise be introduced in the corresponding
affine sections.

\section{Shrinking target sets}\label{sec:shrinking}

In this section we survey results about Hausdorff measure and Hausdorff dimension  for shrinking target sets. The discussion is organised according to the underlying system. We begin with conformal systems. Then we turn to irrational rotations, where the dimension theory is closely connected with inhomogeneous Diophantine approximation and continued fractions. Finally, we consider affine iterated function systems and matrix transformations, where non-conformality plays a prominent role.

The notation for Hausdorff measures, ergodic sums and topological pressure is fixed in Section~\ref{sec:preliminaries}. To precisely state the dimensional results, we will introduce in the suitable subsections three notions of pressure: one defined through partition sums over cylinders for expanding Markov maps, one associated with conformal iterated function systems, and the singular value pressure for affine systems.

\subsection{Conformal systems}\label{sec:conformal-shrinking}

We first consider three classes of conformal systems: expanding
rational maps, expanding Markov maps in the unit interval and conformal iterated function systems. We then describe a general dimension transference principle
which unifies several of these settings.

\subsubsection{Expanding rational maps}

Let $T:\widehat{\mathbb C}\to\widehat{\mathbb C}$ be an expanding rational map 
 of degree $d \ge 2$ of the Riemann sphere $\widehat{\mathbb C}= \mathbb C \cup \{\infty \}$ and $J$ be its associated Julia set. Write $ \delta=\dim_{\mathrm H}J.$ We use the spherical metric on $\widehat{\mathbb C}$, write $B(z,r)$ for
the corresponding ball, and let $T'$ denote the derivative of $T$. 

%For $\tau\geq0$, put
%\[\psi_\tau(n):=e^{-n\tau}\quad\text{and}\quad \vartheta_\tau(n,z):=|(T^n)'(z)|^{-\tau}. \]
Let $\psi:\mathbb{N}\to (0,\infty) $ be a positive decreasing function with $\psi(n)\to 0$ as $n\to\infty$, and $\alpha(\psi)$ be its lower exponential order at infinity. Put
\[\vartheta_\psi(n,z):=\psi(|(T^n)'(z)|).\]
Given $z_0\in J$, define
\[
   W(T,z_0,\psi):=\left\{z\in J:T^n z\in B(z_0,\psi(n))\text{ for infinitely many }n\right\},
\]
whereas 
\[
   W(T,z_0,\vartheta_\psi):=\left\{ z\in J:T^n z\in B(z_0,\psi(|(T^n)'(z)|))\text{ for infinitely many }n \right\}.
\]

The first dimension estimates for these sets were obtained by Hill and Velani in \cite{HillVelani1995}. Let $\mu_\delta$ be the equilibrium measure for $T$ with respect to the potential  $-\delta\log|T'|$, and denote
\[
 \chi_\delta:=\int_J\log|T'(z)|\,\td\mu_\delta(z).
\]
 If $t(\psi)$ is the unique solution of
\[
 P\bigl(T,-s\log|T'|\bigr)=s\alpha(\psi),
\]
{where $P(T,\cdot)$ is defined in \eqref{classic-pressure},} then
\begin{equation}\label{eq:rational-initial-bounds}
   \frac{\delta\chi_\delta}{\chi_\delta+\alpha(\psi)}\leq\dim_{\mathrm H}W(T,z_0,\psi)\leq t(\psi).
\end{equation}
Let
\[
\kappa(\psi):=\liminf_{n\to\infty}\frac{-\log\psi(n)}{\log n},
\]
which is the lower polynomial order at infinity of $\psi$. The same work also obtained that 
\begin{equation}\label{eq:hill-velani-derivative-adapted}
\dim_{\mathrm H} W(T,z_0,\vartheta_\psi) =\frac{\delta}{1+\kappa(\psi)}.
\end{equation}

%The exact dimension of $W_{z_0}(\tau)$ was obtained in their subsequent work \cite{HillVelani1997}.

Hill and Velani  \cite{HillVelani1997} subsequently obtained a pressure formula for a different generalisation, in which the shrinking rate is generated by the ergodic sums of a fixed potential. Let $f:J\to\mathbb R$ be a H\"older continuous function satisfying $ f(z)\geq\log|T'(z)|$ for all $z\in J$.
  For $z_0\in J$, define
\[
 D(T,z_0,f) := \left\{ z\in J: z\in B\bigl(y,e^{-S_nf(y)}\bigr) \ \text{for infinitely many pairs }(y,n) \text{ with }T^ny=z_0 \right\}.
\]
\begin{theorem}[Hill--Velani {\cite{HillVelani1997}}]
\label{thm:hill-velani-rational-pressure}
Let $T$, $J$, $z_0$ and $f$ be as above. Then for $z_0\in J$,
\[
   \dim_{\mathrm H}D(T,z_0,f)=s(f),
\]
where  $s(f)$ is the unique solution of the pressure equation
\begin{equation}
    P\bigl(T,-s f\bigr)=0. \label{eq:rational-general-pressure}
\end{equation}

\end{theorem}

%The connection between Theorem~\ref{thm:hill-velani-rational-pressure} and the forward shrinking target problem follows from bounded distortion.
For $\tau\ge0$, taking $f_\tau=\log|T'|+\tau,$ we have
\[
 e^{-S_nf_\tau(y)}  =\frac{e^{-n\tau}}{|(T^n)'(y)|}.
\]
By the bounded distortion property of $T$, for each $y\in T^{-n}(z_0)$, the intersection of $T^{-n}B(z_0,e^{-n\tau})$ and each cylinder set of level $n$ has diameter comparable to $e^{-S_nf_\tau(y)}$, with comparison constants independent of $n$ and $y$.  Hence
\[\dim_{\mathrm H}W(T,z_0,\psi_\tau)=\dim_{\mathrm H}D(T,z_0,f_\tau),\]
where $\psi_\tau(n)=e^{-n\tau}$. Since
\[
 P\bigl(T,-s(\log|T'|+\tau)\bigr)=0 \quad\Longleftrightarrow\quad P\bigl(T,-s\log|T'|\bigr)=s\tau,
\]
Theorem~\ref{thm:hill-velani-rational-pressure}  provides a precise dimensional formula for $W(T,z_0,\psi_\tau)$ stated as follows. 
\begin{corollary}[Hill--Velani {\cite{HillVelani1997}}]
\label{cor:rational-shrinking-target-dimensions}
For any $z_0\in J$ and $\tau\ge 0$,
\[
   \dim_{\mathrm H}W(T,z_0,\psi_\tau)=t(\psi_\tau).
\]
\end{corollary}
Thus the upper bound in \eqref{eq:rational-initial-bounds} is exact for $\psi_\tau$.
Moreover, when $\psi(n)=n^{-\tau }$, the formula \eqref{eq:hill-velani-derivative-adapted} may be derived  by taking $ f_\tau^{*}=(1+\tau)\log|T'|.$ Then 
$$ e^{-S_nf_\tau^{*}(y)} = |(T^n)'(y)|^{-(1+\tau)},$$ and the pressure equation becomes
\[
 P\bigl(T,-s(1+\tau)\log|T'|\bigr)=0.
\]
Since  $ P\bigl(T,-\delta\log|T'|\bigr)=0,$ its solution is $s=\delta/(1+\tau)$, in agreement with \eqref{eq:hill-velani-derivative-adapted}.

When $f$ satisfies the strict inequality
\[
 f(z)>\log|T'(z)| \quad z\in J,
\]
Hill and Velani \cite{HillVelani2002} subsequently proved 
\[ \mathcal{H}^{s(f)}\bigl(D(T,z_0,f)\bigr) \in\{0,\infty\},\]
where $s(f)$ is the solution of \eqref{eq:rational-general-pressure}.

\subsubsection{Expanding Markov maps}\label{subsubsec:expanding-markov-maps}

Let $\mathcal A$ be a finite or countable alphabet, and assume that $\#\mathcal A\ge 2$. Let $\{I_a\}_{a\in\mathcal A}$ be a  family of  disjoint subintervals of $[0,1]$ with non-empty interiors. For $\omega=(\omega_1,\ldots,\omega_n)\in\mathcal A^n$, write $ I_\omega= \bigcap_{k=0}^{n-1}T^{-k}I_{\omega_{k+1}}.$  Following \cite{Reeve2011}, a map
\[
 T:\bigcup_{a\in\mathcal A}I_a\longrightarrow[0,1].
\]
is called an \emph{expanding Markov map} if the following conditions hold:
\begin{enumerate}
\item For each ${a\in\mathcal A}$,  the restriction $T|_{I_a}$ is $C^1$ and maps the interior of $I_a$ onto $(0,1)$.

\item There exist $\xi>1$ and $N\ge1$ such that $$|(T^n)'(x)|>\xi^n$$
for all $n\ge N$ and all $x\in \bigcup_{\omega\in {\mathcal A}^n}I_{\omega}$. 

\item There exists some sequence $\{\rho_n\}_n\subset \mathbb R$ with $\rho_n\to0$ such that for all $n\ge1$, $\omega\in {\mathcal A}^n$ and $x,y\in I_{\omega}$,
\[e^{-n\rho_n}\le\frac{|(T^n)'(x)|}{|(T^n)'(y)|}\le e^{n\rho_n}\]
\end{enumerate}
Condition~(1) is the full-branch property.  We say that $T$ has \emph{finite branches} if $\#\mathcal{A}<\infty$, whereas \emph{countable branches} if  $\mathcal{A}$ is countably infinite.
%The alphabet \(\mathcal A\) may be finite or countable.  For an admissible word \(\omega\) of length \(n\), denote the corresponding cylinder by \(I_\omega\).

The class of expanding Markov maps defined above includes several familiar examples. The map $x\mapsto bx\pmod 1$, with $b\ge2$ an integer, has finite branches, whereas the Gauss map and $\alpha$-L\"uroth map, including the classical L\"uroth map, have countably many full branches. 

The repeller of $T$ is defined as
\[
 \Lambda:=\bigcap_{n\geq0}T^{-n}([0,1]).
\]
For a function $g:\Lambda\to\mathbb{R}$, define its $n$-th level variation by
\[
 \operatorname{var}_n(g): = \sup_{\omega\in\mathcal A^n} \sup_{x,y\in I_\omega} |g(x)-g(y)|.
\]
We say that $g$ satisfies the \emph{tempered distortion condition} if 
\[
 \operatorname{var}_1(g)<\infty \quad\text{and}\quad \lim_{n\to\infty} \frac{\operatorname{var}_n(S_ng)}{n}=0.
\]
Condition~(3) implies that $ \log|T'|$ satisfies the tempered distortion condition.

Let $\varphi:\Lambda\to[0,\infty)$ satisfy the tempered distortion
condition and set
\[
   \psi_\varphi(n,x):=\exp\bigl(-S_n\varphi(x)\bigr).
\]
For $y\in \overline{\Lambda}$, the corresponding shrinking target set is
\begin{equation*}\label{eq:markov-potential-target}
 W(T,y,\psi_\varphi)= \left\{x\in\Lambda: T^nx\in B(y,\psi_\varphi(n,x)) \ \text{for infinitely many }n
 \right\}.
\end{equation*}
If $\varphi\equiv\alpha\ge0$, then $\psi_\varphi(n,x)=e^{-n\alpha}$. In this case we write $\psi_\alpha(n):=e^{-n\alpha}$. 

The pressure used here is defined directly through  partition sums over cylinders. For a real-valued function $g$ satisfying  tempered distortion condition, the pressure function is defined by
\begin{equation}\label{eq:markov-pressure}
 P(g):= \lim_{n\to\infty}\frac{1}{n} \log \sum_{\omega\in\mathcal A^n} \exp\left(   \sup_{x\in I_\omega}S_ng(x) \right).
\end{equation}
When $\mathcal A$ is countable, the sum and the pressure may be infinite. Define
\begin{equation*}\label{eq:markov-critical-exponent}
 s(\varphi): =\inf\left\{ s\ge0: P\bigl(-s(\log|T'|+\varphi)\bigr)\le0 \right\}.
\end{equation*}

For finitely many branches, the dimension formula is the interval map counterpart of the result of Hill and Velani for expanding rational maps
\cite{HillVelani1997}; see also \cite[Theorem~1]{Reeve2011}. 
%{\color{red} An earlier result of Urba\'nski \cite{Urbanski2002}, formulated for countable conformal iterated function systems, can be applied to the inverse branches of a countable expanding Markov map. His result implies that
%\[
%   \dim_{\mathrm H}W(T,y,\psi_\varphi)=s(\varphi)
%\]
%for  $y\in\Lambda$ satisfying $\sup_{n\geq0}\bigl|T'(T^n y)\bigr|<\infty$; see also \cite[Theorem~2]{Reeve2011}.}
An earlier result of Urba\'nski \cite{Urbanski2002} established the corresponding dimension formula for countable conformal iterated function systems. Applied to the inverse branches of a countable expanding Markov map, it gives
\[\dim_{\mathrm H}W(T,y,\psi_\varphi)=s(\varphi)\]
for any $y\in\Lambda$ admitting a symbolic coding that uses only finitely many inverse branches; see also \cite[Theorem~2]{Reeve2011}. Reeve \cite{Reeve2011} subsequently extended the formula to each $y\in\Lambda$.

 \begin{theorem}[Reeve {\cite{Reeve2011}}]
\label{thm:reeve-countable-markov}
Let $T$ be an expanding Markov map with repeller $\Lambda$, and let
$\varphi:\Lambda\to[0,\infty)$ satisfy the tempered distortion
condition. Then, for any $y\in\Lambda$,
\begin{equation}\label{eq:reeve-dimension-formula}
 \dim_{\mathrm H}W(T,y,\psi_\varphi)=s(\varphi).
\end{equation}
If, in addition, $\overline{\Lambda}=[0,1]$, then \eqref{eq:reeve-dimension-formula} holds for every $y\in[0,1]$.
\end{theorem}
 The location of the target centre $y$ becomes important when $\overline{\Lambda}\neq[0,1]$. Reeve showed that, for any strictly decreasing function $\Psi:\mathbb N\to(0,\infty)$ tending to zero, there exist an expanding Markov map $T$ with repeller $\Lambda$ and a point $y\in\overline{\Lambda}\setminus\Lambda$ such that
 {\[\dim_{\mathrm H}\left(\limsup_{n\to\infty}\{x:T^nx\in B(y,\Psi(n))\}\right)=0.\]
 In particular, for $\Psi=\psi_\alpha$, the shrinking target set may have dimension zero even though $s(\alpha)$ is strictly positive.} Thus, \eqref{eq:reeve-dimension-formula} cannot in general be extended from $y\in\Lambda$ to every $y\in\overline{\Lambda}$.
 
 %Reeve constructed examples in which $y\in\overline{\Lambda}\setminus\Lambda$ and 
%\[
% \dim_{\mathrm H}W(T,y,\psi)=0
%\]
%although $s(\varphi)>0$. Consequently, \eqref{eq:reeve-dimension-formula} does not extend from $y\in\Lambda$ to every point of $\overline{\Lambda}$ without an additional assumption such as $\overline{\Lambda}=[0,1]$.

%The class above contains several special cases.The results related to cylinder-target for $\alpha$-L\"uroth transformations were proved by Shen and Li \cite{ShenLi2016}.  

%Related work has considered more general target schemes and other classes of interval maps. 
{Related work has allowed the target centres and radii to vary more generally and has treated other classes of interval maps.} For the Gauss map, Li, Wang, Wu and Xu \cite{LiWangWuXu2014} allowed the target centre to vary with $n$, and
 obtained Hausdorff dimension formulae for shrinking rates depending only on $n$ and for rates determined by ergodic sums. Shen and Li \cite{ShenLi2016} determined the Hausdorff dimension for $\alpha$-L\"uroth transformations when the targets are cylinders containing a fixed point. For non-integer $\beta>1$, the $\beta$-transformation is generally not full-branch, and related results will be discussed in Section~3.4. For more general expanding systems, Fernández, Melián and Pestana \cite{FernandezMelianPestana2007} obtained bounds estimates on the  dimension of shrinking target sets for almost all target centre. Their later work \cite{FernandezMelianPestana2012} focused on Borel--Cantelli results on hitting times for uniformly expanding maps and certain non-uniformly expanding maps.

\subsubsection{Conformal iterated function systems}
The precise definition of a conformal iterated function system
varies slightly in the literature. For the finite systems below, we use the following common framework.%For the finite systems considered in the first part of this subsection, we use the following common framework. It retains the regularity assumptions on the domain used by Allen and B\'ar\'any \cite{AllenBarany2021} and makes explicit the neighbourhood extension appearing in Baker \cite{Baker2019SelfConformal}. This convention is adopted for uniformity and is not intended to give the weakest possible hypotheses.

Let $d\ge1$,  $D\subset\mathbb R^d$ be a compact, simply connected set with non-empty interior such that  $D=\overline{D^\circ}$, and let $V$ be an open connected neighbourhood of $D$. 

 A map $f:V\to\mathbb R^d$ is said to be of class $C^{1+\varepsilon}$ if $f$ is continuously differentiable and there exists $c>0$ such that 
\[
   | f'(\bfx)-f'(\bfy)| \le c | \bfx-\bfy |^\varepsilon
\]
for all $\bfx,\bfy\in V$. Here, we adopt the notation that $f'(\bfx)$ is the Jacobian matrix of $f$ at $\bfx$.  A $C^1$ map $f:V\to\mathbb R^d$ is called conformal if, for
all $\bfx\in V$ and $\bfy \in \mathbb R^d$,
 $$|f'(\bfx)\bfy| = |f'(\bfx)|\, |\bfy|.$$

Let $\mathcal A$ be a finite alphabet with $\#\mathcal A\geq2$. 
A finite $C^{1+\varepsilon}$ conformal iterated function system is a
family
\[
   \Phi=\{\varphi_i:D\to D\}_{i\in\mathcal A}
\]
of strict contractions such that each $\varphi_i$ extends to an
injective $C^{1+\varepsilon}$ conformal map on $V$, and there exist
constants $0<a_{\min}\leq a_{\max}<1$ such that for all $i\in\mathcal A$,
\[ a_{\min}\leq \inf_{\bfx\in V}|\varphi_i'(\bfx)|\leq\sup_{\bfx\in V}|\varphi_i'(\bfx)|\leq a_{\max}.\]
The attractor of $\Phi$ is the unique non-empty compact set $X$
satisfying
\[
   X=\bigcup_{i\in\mathcal A}\varphi_i(X).
\]

For a word $\mathbf i=(i_1,\ldots,i_n)\in\mathcal{A}^n$, write
\[ \varphi_{\mathbf i}:=\varphi_{i_1}\circ\cdots\circ\varphi_{i_n},\quad X_{\mathbf i}:=\varphi_{\mathbf i}(X),\]
and set
\[
   |\varphi_{\mathbf i}'|:=\sup_{\bfx\in D}|\varphi_{\mathbf i}'(\bfx)|.
\]
%The bounded distortion property gives a constant $C\geq1$ such that
%\[
 %  C^{-1}\leq\frac{|\varphi_{\mathbf i}'(x)|}{|\varphi_{\mathbf i}'(y)|}
%   \leq C
%\]
%for any $\mathbf i\in\mathcal A^*$ and $x,y\in D$. In particular, $\diam(X_{\mathbf i})$ is comparable to $\|\varphi_{\mathbf i}'\|$, with comparison constants independent of $\mathbf i$.
The bounded distortion property yields that $\diam(X_{\mathbf i})$ is comparable to $|\varphi_{\mathbf i}'|$, with comparison constants independent of $\mathbf i$. Consequently, the pressure defined below is unchanged if $|\varphi_{\mathbf i}'|$ is replaced by $\diam(X_{\mathbf i})$.

For $s\ge0$, define the pressure 
\begin{equation*}\label{eq:self-conformal-pressure}
 P_\Phi(s):=\lim_{n\to\infty}\frac{1}{n}
 \log\sum_{\mathbf{i}\in\mathcal A^n}|\varphi_{\mathbf{i}}'|^s.
\end{equation*}
The limit exists due to the submultiplicativity of the map $n\mapsto \sum_{\mathbf{i}\in\cA^n}|\varphi_{\mathbf{i}}'|^s$. The conformality dimension of the conformal IFS $\Phi$ is defined as $\delta_\Phi:=P_{\Phi}^{-1}(0)$.

We say that $\Phi$ satisfies the \emph{open set condition} if there exists a non-empty open set $U\subset D$ such that
\[ \varphi_i(U)\subset U \quad\text{for every }i\in\mathcal A, \quad \varphi_i(U)\cap\varphi_j(U)=\varnothing \quad\text{whenever }i\neq j.
\]
Under this condition,  $\dimH X=\delta_\Phi$ and $ 0<\cH^{\delta_\Phi}(X)<\infty.$

 Given  $\bfz\in X$ and $\Psi:\mathcal A^*\to(0,\infty)$, the associated shrinking target set is
\begin{equation*}\label{eq:self-conformal-target}
 W(\Phi,\bfz,\Psi):=\left\{ \bfx\in X: \bfx\in B (\varphi_{\mathbf i}(\bfz),\Psi(\mathbf i)) \ \text{for infinitely many } \mathbf i\in\mathcal A^* \right\}.
\end{equation*}
A principal class of radii is obtained from a decreasing function
$\psi:\mathbb N\to(0,\infty)$ by setting
\[
   \Psi_\psi(\mathbf i):=\diam(X_{\mathbf i})\psi(|\mathbf i|).
\]

Baker \cite{Baker2019SelfConformal}  obtained the convergence statement for arbitrary radii and the full measure divergence statement for radii of the form
$\Psi_\psi$. In particular, under the open set condition, 
\begin{equation}\label{eq:baker-self-conformal}
 \sum_{\mathbf i\in\mathcal A^*}\Psi_\psi(\mathbf i)^{\delta_\Phi}=\infty
 \quad\Longrightarrow\quad
 \cH^{\delta_\Phi}\bigl( W(\Phi,\bfz,\Psi_\psi)\bigr)=\cH^{\delta_\Phi}(X).
\end{equation}
Combining this result with the mass transference principle gives that  for $s\le \delta_\Phi$, 
\begin{equation*}\label{eq:baker-self-conformal-an}
 \sum_{\mathbf i\in\mathcal A^*}\Psi_\psi(\mathbf i)^{\delta_\Phi}=\infty\quad\Longrightarrow\quad\cH^s\bigl(W(\Phi,\bfz,\Psi_\psi^{\,\delta_\Phi/s})\bigr)=\cH^s(X).
\end{equation*}
This statement concerns the modified radii $\Psi_\psi(\mathbf i)^{\delta_\Phi/s}$, rather than the original radii $\Psi_\psi(\mathbf i)$.

Allen and B\'ar\'any \cite{AllenBarany2021} obtained the zero-full law for Hausdorff measure for the original radii. 

\begin{theorem}[Allen--B\'ar\'any {\cite{AllenBarany2021}}]
\label{thm:allen-barany-conformal-ifs}
Suppose that $\Phi$ satisfies the open set condition and $\psi:\mathbb N\to(0,\infty)$ is decreasing. Then, for any $s\geq0$, $\bfz\in X$ and open ball $B\subset\mathbb R^d$,
\begin{equation*}\label{eq:allen-barany-self-conformal}
 \cH^s\bigl(W(\Phi,\bfz,\Psi_\psi)\cap B\bigr)
 =\begin{cases}
 0,&\displaystyle \text{if }
  \sum_{\mathbf i\in\mathcal{A}^*}\Psi_\psi(\mathbf i)^s<\infty,\\[2ex]
 \cH^s(X\cap B),
 &\displaystyle \text{if }
  \sum_{\mathbf i\in\mathcal{A}^*}\Psi_\psi(\mathbf i)^s=\infty.
 \end{cases}
\end{equation*}
\end{theorem}

Baker \cite{BakerOverlaps2023} considered finite conformal iterated
function systems without the open set condition.  Let $\mu_\Phi$ be the projection onto $X$ of the equilibrium measure associated with $\delta_\Phi$.  
For $z\in X$ and a decreasing function $\psi:\mathbb N\to[0,\infty)$, he proved that
\begin{equation}\label{eq:baker-overlap-divergence}
\sum_{\mathbf i\in\mathcal A^*}
   \Psi_\psi(\mathbf i)^{\delta_\Phi} = \infty \quad\Longrightarrow\quad \mu_\Phi\bigl(W(\Phi,\bfz,\Psi_\psi)\bigr)=1.
\end{equation}
When the open set condition holds, $\delta_\Phi=\dim_{\mathrm H}X$ and $\mu_\Phi$ is equivalent to $\mathcal H^{\dim_{\mathrm H}X}|_X$. Thus \eqref{eq:baker-overlap-divergence} recovers the corresponding full measure statement \eqref{eq:baker-self-conformal} in \cite{Baker2019SelfConformal} as a special case.
If $\dim_{\mathrm H}X=\inf\bigl\{\dim_{\mathrm H}E:E\subset \mathbb{R}^d\ \text{is Borel and }\mu_\Phi(E)>0\bigr\}, $ then \eqref{eq:baker-overlap-divergence} implies
\[ \dim_{\mathrm H}W(\Phi,\bfz,\Psi_\psi)=\dim_{\mathrm H}X. \]

Daviaud \cite{DaviaudWeakConformal2025} extended the dimension theory to  finite  $C^1$ weakly conformal iterated function systems with overlaps. Here weak conformality means that the ratio between the largest and
smallest singular values of $\varphi_{\mathbf i}'$ grows
subexponentially with $n$, uniformly over $\mathbf i\in\mathcal A^n$.
In the conformal case considered here, his result takes the following form. For $\eta>0$, put
\[
   \Psi_\eta(\mathbf i):=\diam(X_{\mathbf i})^\eta.
\]
Suppose that $\delta_\Phi=\dim_{\mathrm H}X.$ Then, for any $\bfz\in X$ {and $\eta\ge1$},
\begin{equation*}\label{eq:daviaud-weak-conformal}
 \dim_{\mathrm H}W(\Phi,\bfz,\Psi_\eta) = \frac{\dim_{\mathrm H}X}{\eta}.
\end{equation*}
No separation condition is required for this conclusion.
Daviaud also considered radii of the more general form
\[
   \Psi_{\eta,\psi}(\mathbf i):=\diam(X_{\mathbf i})^\eta\psi(|\mathbf i|).
\]
    Under the bounded distortion property and an additional dimension assumption on a suitable projected Gibbs measure, he obtained a Hausdorff dimension formula; see \cite[Theorem~3.8]{DaviaudWeakConformal2025}. When $\eta=1$, these radii reduce to $\Psi_\psi(\mathbf i)$.

Countable conformal IFSs were studied earlier by Urba\'nski
\cite[Theorem~1.7]{Urbanski2002}. Under the standing assumptions of
that theorem, including the open set condition, the required
distortion properties and the inclusion of the limit set in the
interior of the ambient domain, let $F=\{f^{(i)}:i\in \mathcal A\}$ be a strongly H\"older family {satisfying a domination condition $f\geq\xi$, where $f$ is the potential function induced by $F$ and $\xi$ by the derivative family $\{-\log|\varphi_i'|:i\in\mathcal A\}$.}  Given  $\mathbf i=i_1\cdots i_n\in\mathcal A^n$,  write  $S_{\mathbf i}F =\sum_{k=1}^{|\mathbf i|} f^{(i_k)}\circ\varphi_{\sigma^k(\mathbf i)}$, where $\varphi_{\varnothing}$ denotes the identity map. Define%Its pressure is
\[P(F):=\lim_{n\to\infty}\frac{1}{n}\log\sum_{\mathbf i\in\mathcal A^n} \exp\left(
      \sup_{\bfx\in X}S_{\mathbf i}F(\bfx)\right).\]
For $\bfz\in X$ and $C>0$, let $ \Psi_{F,\bfz,C}(\mathbf i):=C\exp\bigl(-S_{\mathbf i}F(\bfz)\bigr).$ Define
\[
   W(\Phi,\bfz,\Psi_{F,\bfz,C}):=\left\{ \bfx\in X: \bfx\in B\bigl(\phi_{\mathbf i}(\bfz),\Psi_{F,\bfz,C}(\mathbf i) \bigr)\text{ for infinitely many }\mathbf i\in\mathcal A^*\right\}.
\]
Urba\'nski obtained that  { for every finitely accessible $\bfz$}
\[
 \dim_{\mathrm H} {W(\Phi,\bfz,\Psi_{F,\bfz,C})} = \delta(F), 
 \]
where  $\delta(F)=\inf\{t\ge0:P(-{tf})\le0\}$. Without the corresponding accessibility assumptions on the target centre $\bfz$, only an upper bound is
available in general.

Moving beyond ball targets,  Allen, Jordan and Ward \cite{AllenJordanWard2026} studied rectangular targets for the natural expanding map on a class of planar self-similar carpets with a grid structure. They obtained the Hausdorff dimension formulae, which may depend on the centre of the target.

 %Levesley, Salp and Velani \cite{LevesleySalpVelani2007} established a zero--full law for points in the middle-third Cantor set approximated by rational numbers with denominators $3^n$. Li, Li and Wu \cite{LiLiWu2023} considered the corresponding problem for general missing digit sets, and showed that the metric approximation behaviour depends essentially on the  relation between the base defining the missing digit set and the prescribed denominators, obtaining zero--full laws in some cases.
A closely related Diophantine problem concerns approximation inside
missing digit sets. 
{For an integer $t\ge 2$ and a positive function $\psi$, let $W_t(\psi)$ denote the set of points in $[0,1]$ that are $\psi$-well approximable by rational numbers with denominators $t^n$. 
Levesley, Salp and Velani~\cite{LevesleySalpVelani2007}
established a zero--full law for the Hausdorff measure of $W_3(\psi)\cap C_{1/3}$, where $C_{1/3}$ is the middle-third Cantor set. They also considered the analogous problem for $W_b(\psi)\cap C(b,D)$ for general missing digit sets $C(b,D)$. Li, Li and Wu~\cite{LiLiWu2023}  allowed the base $b$ defining $C(b,D)$ and the base $t$ in $W_t(\psi)$ to differ, and obtained a zero-full law for the Hausdorff measure in several cases. They also proposed a product conjecture for the dimension of the intersection. More recently, Li, Velani and Wang \cite{LiVelaniWang} studied $W_t(\psi)\cap S$ for an arbitrary self-similar set $S$. They proved a general upper bound for its Hausdorff dimension and, under additional restrictions on $D$, a sufficient measure-zero criterion that closes a gap in \cite{LiLiWu2023}. Consequently, they showed that $\dimH(W_t(\psi)\cap C(b,D))$ can be strictly smaller than $\dimH W_t(\psi) \dimH C(b,D)$, disproving  the product conjecture.  } %When $b$ and $t$ are multiplicatively dependent, they  proved a zer, ao--full law; in a multiplicatively independent case, they obtained partial Hausdorff-measure results and bounds for the Hausdorff dimension.

\subsubsection{A dynamical dimension transference principle}
\label{subsubsec:dynamical-dimension-transference}

Wang and Zhang \cite{WangZhang2021} provided a dynamical dimension transference principle which unifies and extends several of the conformal shrinking target results discussed above. Let $(X,d)$ be
a compact metric space, let $T:X\to X$ be continuous, and fix
$y_0\in X$. Let $g:X\to(0,\infty)$ be {the continuous function occurring in the local conformality
condition}, and let $f:X\to[0,\infty)$ be continuous. Define
\[
 W(T,y_0,\psi_{f+g}): =\left\{ x\in X: x\in B(z,e^{-S_n(f+g)(z)})\text{ for infinitely many pairs} (z,n)\text{ with } T^nz=y_0 \right\}.
\]
They \cite{WangZhang2021} assumed three conditions for the system: dynamical ubiquity; local conformality; exactness. Under these hypotheses, one has
\[
 \dim_{\mathrm H}X = \inf\{t\geq0:P(T,-tg)\leq0\},\]
and
\[ \dim_{\mathrm H} W(T,y_0,\psi_{f+g}) = \inf\{t\geq0:P(T,-t(g+f))\leq0\}.
\]
Fix $z\in T^{-n}\{y_0\}$. In the differentiable conformal case, $g=\log|T'|$, the natural scale of the corresponding $n$-th inverse branch is $e^{-S_ng(z)}$. The shrinking condition introduces the additional factor $e^{-S_nf(z)}$. Hence the resulting radius is
\[
e^{-S_ng(z)}e^{-S_nf(z)}=e^{-S_n(g+f)(z)}.
\]
Hence the pressure formula involves the sum of the  potentials $g$ and $f$.

The hypotheses of \cite[Theorem~1.4]{WangZhang2021} are satisfied
in several standard settings, including:
\begin{itemize}
\item  $b$-adic transformations: $T:[0,1)\to[0,1): x\mapsto bx~(\bmod{1})$, where $b \ge 2$ is an integer;
\item expanding rational maps $T$ with degree at least $2$, acting on
their Julia sets $X$;
\item the natural expanding map $T:C_{1/3}\to C_{1/3}: x\mapsto 3x~(\bmod1)$ on the middle-third Cantor set $C_{1/3}$;
\item finite conformal iterated function systems satisfying the strong
separation condition.
\end{itemize}

Wang and Zhang \cite{WangZhang2021} also observed that the dimension formula of Urba\'nski \cite{Urbanski2002} can be deduced from their theorem by applying it to suitable finite subsystems and using the approximation property of the pressure function. This is possible because the target centre in Urba\'nski's result has a symbolic coding involving only finitely many symbols.

\subsection{Irrational rotations}

Let $\mathbb T=\mathbb R/\mathbb Z$, equipped with the metric
\[
   \|x\|:=\min_{k\in\mathbb Z}|x-k|,
\]
and let $m$ denote normalized Lebesgue measure on $\mathbb T$.
For an irrational number $\alpha$, consider the rotation
\[
   R_\alpha(x):=x+\alpha \pmod 1.
\]
For $\psi:\mathbb N\to(0,\infty)$ and  $z\in\mathbb T$, define
\begin{equation*}\label{eq:rotation-shrinking-target}
 W(R_\alpha,z,\psi):= \left\{ x\in\mathbb T: \|x+n\alpha-z\|<\psi(n) \ \text{for infinitely many }n \right\}.
\end{equation*}
It is convenient to work with the equivalent set
\begin{equation*}\label{eq:rotation-orbit-approximation}
 E_\psi(\alpha):=\left\{y\in\mathbb T:\|n\alpha-y\|<\psi(n)\ \text{for infinitely many }n\right\}.
\end{equation*}
Indeed, the map $x\mapsto z-x$ is an isometry from $W(R_\alpha,\psi,z)$ onto $E_\psi(\alpha)$.  Thus the Lebesgue measure,  Hausdorff measure and dimension of $W(R_\alpha,z,\psi)$ are independent of $z$.

We briefly recall that the Lebesgue-measure theory is well understood. Cassels~\cite{Cassels1950} proved that $m\bigl(W(R_\alpha,z,n\mapsto c/n)\bigr)=1
$ for any irrational $\alpha$ and  $c>0$. Kurzweil~\cite{Kurzweil1955} proved that an irrational number $\alpha$ is badly approximable if and only if, for every positive decreasing function $\psi$ satisfying $\sum_{n=1}^{\infty}\psi(n)=\infty,$ one has $m\bigl(W(R_\alpha,z,\psi)\bigr)=1$. Fuchs and Kim~\cite{FuchsKim2016} gave a complete criterion for each fixed irrational $\alpha$ and positive non-increasing $\psi$.

%Now we focus on the Hausdorff dimension and Hausdorff measure of $W(R_\alpha,z,\psi)$. 
We now turn to the Hausdorff dimension of $W(R_\alpha,z,\psi)$. The Diophantine type of $\alpha$ is defined by 
\[
 \beta(\alpha) := \sup\left\{ \theta\ge1: \liminf_{n\to\infty}n^\theta\|n\alpha\|=0 \right\}.
\]
For $\psi_\gamma(n):=n^{-\gamma}$, $\gamma\geq1$, Bernik and Dodson \cite{BernikDodson1999} proved that
\[
 \frac{1}{\beta(\alpha)\gamma}\leq\dim_{\mathrm H}W(R_\alpha,z,\psi_\gamma)\leq\frac{1}{\gamma}.
\]
The exact value was subsequently obtained independently by Bugeaud
\cite{Bugeaud2003} and Schmeling and Troubetzkoy
\cite{SchmelingTroubetzkoy2003}.
\begin{theorem}[{Bugeaud; Schmeling--Troubetzkoy
\cite{Bugeaud2003,SchmelingTroubetzkoy2003}}]
\label{thm:rotation-power-law} 
For any irrational $\alpha$, $z\in\mathbb T$ and $\gamma>1$,
\begin{equation*}\label{eq:rotation-power-law}
 \dim_{\mathrm H}W(R_\alpha,z,\psi_\gamma)=\frac{1}{\gamma}.
\end{equation*}
\end{theorem}
\noindent For $\gamma=1$, the full-measure result of Cassels gives $\dim_{\mathrm H} W(R_\alpha,z,\psi_1)=1.$

For a general decreasing function $\psi:\mathbb N\to(0,\infty)$ with $\psi(n)\to0 $, set
\begin{equation*}\label{eq:rotation-decay-exponents}
 d_-(\psi):=\liminf_{n\to\infty}\frac{\log n}{-\log\psi(n)},
 \quad   d_+(\psi):=\limsup_{n\to\infty} \frac{\log n}{-\log\psi(n)}.
\end{equation*}
Under the additional assumption that $n\psi(n)\le1/2$ for all sufficiently large $n$, Fan and Wu \cite{FanWu2006} proved that
\begin{equation*}\label{eq:fan-wu-rotation}
  d_-(\psi)\le \dim_{\mathrm H}W(R_\alpha,z,\psi)\le d_+(\psi).
\end{equation*}
They also showed that the upper bound is attained for Lebesgue-almost all $\alpha$, and for all $\alpha$ with bounded partial quotients. They also found a Liouville number $\alpha$ and a decreasing $\psi$ such that
\[
 \dim_{\mathrm H}W(R_\alpha,z,\psi)=d_-(\psi)=0<d_+(\psi).
\]

Let $(q_k)_{k\geq0}$ be the denominators of the convergents of the continued fraction of $\alpha$. Xu \cite{Xu2010} proved that
\begin{equation*}\label{eq:xu-rotation}
 \limsup_{k\to\infty}\frac{\log q_k}{-\log\psi(q_k)}\le\dim_{\mathrm H}W(R_\alpha,z,\psi)\le d_+(\psi).
\end{equation*}
In particular, if $\beta(\alpha)=1$, then
\[
 \dim_{\mathrm H}W(R_\alpha,z,\psi)=d_+(\psi).
\]

Liao and Rams \cite{LiaoRams2013} obtained a sharper estimate
\begin{equation*}\label{eq:liao-rams-rotation}
 \min\left\{d_+(\psi), \max\left\{d_-(\psi),\frac{1+d_+(\psi)}{1+\beta(\alpha)} \right\}\right\}
   \leq
   \dim_{\mathrm H}W(R_\alpha,z,\psi)
   \leq
   d_+(\psi).
\end{equation*}
They also constructed examples attaining the lower and upper bounds, respectively, thereby showing that both estimates are sharp.

A complete dimension formula  was obtained by Kim, Rams and Wang \cite{KimRamsWang2018}. 
\begin{theorem}[{Kim-Rams-Wang\cite{KimRamsWang2018}}]
\label{thm:kim-rams-wang-rotation}
For $0\le s\le 1$ and $k\ge0$, let $q_{k,s}$ be the integer $m$ with $m\in [q_k,q_{k+1})$ minimizing
\begin{equation*}\label{eq:krw-finite-scale-cost}
 q_k \left(\psi(q_k)+\frac{m-q_k}{q_k}\|q_k\alpha\|\right)^s+\sum_{n=m+1}^{q_{k+1}-1}\psi(n)^s.
\end{equation*}
Then, for any $z\in\mathbb T$,
\begin{equation*}\label{eq:krw-rotation-dimension}
 \dim_{\mathrm H}W(R_\alpha,z,\psi)=\inf\left\{s>0:\sum_{k=0}^{\infty}\left[q_k\left(\psi(q_k)+\frac{q_{k,s}-q_k}{q_k}\|q_k\alpha\|\right)^s+\sum_{n=q_{k,s}+1}^{q_{k+1}-1}\psi(n)^s\right]<\infty\right\}.
\end{equation*}
\end{theorem}

To the best of our knowledge, no complete Hausdorff-measure criterion is known for arbitrary $\alpha$ and $\psi$. For badly approximable $\alpha$, Kurzweil's theorem and the mass transference principle \cite{BeresnevichVelani2006,Kurzweil1955} give, for $0<s\leq1$,
\begin{equation}\label{eq:badly-approximable-rotation-measure}
 \mathcal H^s\bigl(W(R_\alpha,z,\psi)\bigr) =
 \begin{cases}
 0, &\displaystyle\sum_{n=1}^{\infty}\psi(n)^s<\infty,\\[2ex]
 \mathcal H^s(\mathbb T), &\displaystyle\sum_{n=1}^{\infty}\psi(n)^s=\infty.
 \end{cases}
\end{equation}
See also~\cite[Section~9]{eresnevichBernikDodsonVelani2009} and
\cite[Corollary~1.6]{KimKurzweil2024}.

Kim \cite{KimKurzweil2024} obtained a sufficient condition for full Hausdorff measure without assuming that $\alpha$ is badly approximable. For $\varepsilon>0$, define
\[ L_\varepsilon(\alpha) := \left\{ \ell\ge1: \|q\alpha\|\ge\varepsilon 2^{-\ell} \text{ for every }1\le q<2^\ell \right\}.
\]
For any irrational $\alpha$, there exists $\varepsilon>0$ such that $L_\varepsilon(\alpha)$ is unbounded.  Write
\[ L_\varepsilon(\alpha)=\{\ell_1<\ell_2<\cdots\}.\]
If $0<s\le 1$ and
\begin{equation}\label{eq:kim-rotation-hausdorff-measure}
 \sum_{i=1}^{\infty} 2^{\ell_i}\psi(2^{\ell_i})^s =\infty,
\end{equation}
then 
\[
 \mathcal H^s\bigl(W(R_\alpha,z,\psi)\bigr)=\mathcal H^s(\mathbb T).
\]
When $\alpha$ is badly approximable, one may choose $\varepsilon>0$ so that $L_\varepsilon(\alpha)=\mathbb N$.  Since $\psi$ is decreasing, {by the Cauchy condensation test (see \cite[Theorem 3.27]{MR385023},}\eqref{eq:kim-rotation-hausdorff-measure} is then equivalent to $\sum_{n=1}^{\infty}\psi(n)^s=\infty.$ Thus Kim's result recovers the divergence part of \eqref{eq:badly-approximable-rotation-measure}.

\subsection{Affine iterated function systems}
Using the symbolic notation introduced in Section~\ref{subsec:symbolic-pressure}, let $\mathcal A$ be finite with $\#\mathcal{A}\ge 2$ and
let $\Sigma=\mathcal A^{\mathbb N}$. {Recall that $\sigma:\Sigma\to\Sigma$ denotes the left shift.} We consider affine
iterated function systems
\[
 \Phi_{\boldsymbol t} =\{f_i({\bfx})=A_i {\bfx}+{\bf t}_i:i\in\mathcal A\}
\]
on $\mathbb R^d$, where each $A_i\in GL_d(\mathbb R)$ is contracting, and $\boldsymbol t= ({\bf t}_i)_{i\in\mathcal A}\in (\mathbb R^d)^{\mathcal A}$. The attractor $K_{\boldsymbol t}$ is the unique non-empty compact set
satisfying
\[
 K_{\boldsymbol t} =\bigcup_{i\in\mathcal A}f_i(K_{\boldsymbol t}).
\]
The coding map $\pi_{\boldsymbol t}:\Sigma\to K_{\boldsymbol t}$ is
defined by 
\[
 \pi_{\boldsymbol t}(\omega): =\lim_{n\to\infty} f_{\omega_1}\circ\cdots\circ f_{\omega_n}(0).
\]
For a finite word $\mathbf i=i_1\cdots i_n$, write $ A_{\mathbf i}=A_{i_1}\cdots A_{i_n}.$ For $A\in GL_d(\mathbb R)$, let 
\[
 \sigma_1(A)\ge\cdots\ge\sigma_d(A)
\]
denote its singular values. For $s\geq0$, define the singular value
function by
\begin{equation*}\label{eq:self-affine-singular-value}
 \varphi^s(A):=
 \begin{cases}
  \sigma_1(A)\cdots\sigma_k(A)
  \sigma_{k+1}(A)^{s-k},
   &0\le s<d,\quad k=\lfloor s\rfloor,\\[1ex]
  |\det A|^{s/d},&s\ge d.
 \end{cases}
\end{equation*}
{Here $\lfloor s\rfloor$ denotes the greatest integer not exceeding $s$.} The singular value pressure associated with $\boldsymbol A=(A_i)_{i\in\mathcal A}$ is
\begin{equation*}\label{eq:self-affine-pressure}
P_{\boldsymbol A}(s):=\lim_{n\to\infty}\frac{1}{n}\log\sum_{\mathbf i\in\mathcal A^n}\varphi^s(A_{\mathbf i}).
\end{equation*}
The limit exists by submultiplicativity of $\varphi^s$.  %The zero of $P$ is the affinity dimension.  It gives a natural upper bound for the dimension of the attractor and for many dynamically defined subsets.  Exact formulae, however, usually require separation, suitable conditions on the matrices, or a generic choice of the translation vectors.

 Given a sequence $\boldsymbol\lambda=(\lambda_n)_{n\geq1}$ of finite words, define
\[
   W(\Phi_{\mathbf t},\boldsymbol\lambda):=
   \pi_{\mathbf t}\left\{
      \omega\in\Sigma: \sigma^n\omega\in[\lambda_n]\text{ for infinitely many }n
   \right\}.
\]
If $\mathbf j\in\Sigma$ and $\ell=(\ell_n)_{n\geq1}$ is an increasing sequence of positive integers,  write
\[
   W(\Phi_{\mathbf t},\mathbf j,\ell) :=
   W\bigl(\Phi_{\mathbf t}, (\mathbf j|_{\ell_n})_{n\geq1}\bigr).
\]

Koivusalo and Ram\'irez \cite{KoivusaloRamirez2018} considered targets centred at a fixed symbolic point. Assume that, for $s\geq0$, there exists $D_s>0$ such that
\begin{equation}\label{quasimultiplicativity}
   \varphi^s(A_{\mathbf i\mathbf j})\geq
   D_s\varphi^s(A_{\mathbf i})\varphi^s(A_{\mathbf j})
\end{equation}
for all finite words $\mathbf i,\mathbf j$, { where $\mathbf{ij}$ denotes their concatenation}. For $\mathbf j\in\Sigma$, define the modified pressure by 
\begin{equation*}\label{eq:KR-modified-pressure}
P_{\mathbf j,\ell}(s) :=
 \lim_{n\to\infty} \frac{1}{n}\log\sum_{\mathbf i\in\mathcal A^n}
 \varphi^s\bigl(A_{\mathbf i\mathbf j|_{\ell_n}}\bigr). 
\end{equation*}

\begin{theorem}[Koivusalo--Ram\'irez{\cite{KoivusaloRamirez2018}}]
\label{thm:koivusalo-ramirez-affine}
Suppose that  $\max_i\sigma_1(A_i)<1/2$. Assume that there is some $D_s\in (0,1)$ such that \eqref{quasimultiplicativity} holds. 
If $\lim_{k\to\infty}\frac{1}{k} \log \varphi^s(A_{{\mathbf j}|_{\ell_k}})$ exists in $[-\infty, 0]$,
then with respect to any  ergodic $\sigma$-invariant probability measure, there is almost surely a unique number $s_{\mathrm{KR}}\in \mathbb{R}^+$ satisfying $P_{\mathbf j,\ell}(s_{\mathrm{KR}})=0$, and
\begin{equation*}\label{eq:KR-self-affine-dimension}
 \dim_{\mathrm H} W(\Phi_{\mathbf t},\mathbf j,\ell) =\min\{s_{\mathrm{KR}},d\}
\end{equation*}
for Lebesgue-almost all translation vector $\boldsymbol t$.
\end{theorem}
In particular, when $\lim_n\ell_n/n=0$, the shrinking-target set has the same
Hausdorff dimension as $K_{\boldsymbol t}$ for almost all $\boldsymbol t$. If $\lim_n\ell_n/n=\infty$, its Hausdorff dimension is zero for every symbolic centre $\mathbf j$ and translation $\boldsymbol t$.
%With the assumptions of Theorem \ref{thm:koivusalo-ramirez-affine},if $\lim_{n\to\infty}\ell_n/n$ exists, then almost surely
%\begin{equation*}\label{eq:KR-three-regimes}
% \dim_{\mathrm H}W(\Phi_{\mathbf t},\mathbf j,\ell)
% =\begin{cases}
%  \dim_{\mathrm H}K_{\boldsymbol t},
 %    &\lim_n\ell_n/n=0,\\
 % s_{\mathrm{KR}},
 %    &\lim_n\ell_n/n\in(0,\infty),\\
 % 0,
 %    &\lim_n\ell_n/n=\infty.
 %\end{cases}
%\end{equation*} Koivusalo and Ram\'irez also obtained a zero--full law for a natural Gibbs measure.

B\'ar\'any and Troscheit \cite{BaranyTroscheit2021} allowed the centres of cylinder targets to vary  with $n$. They assumed that, for $s>0$, there exist $C_s>0$ and $K_s\in\mathbb N$ such that, for all finite words
$\mathbf i,\mathbf j$, there exists $\mathbf k\in\mathcal A^{K_s}$ satisfying
\begin{equation}\label{nquasi-multiplicativity condition}
 \varphi^s(A_{\mathbf i\mathbf k\mathbf j})\geq C_s\varphi^s(A_{\mathbf i})\varphi^s(A_{\mathbf j}).
\end{equation}
This condition is similar to, but different from, that given in \cite{KoivusaloRamirez2018}. For a sequence $\boldsymbol\lambda=(\lambda_n)_{n\geq1}$, define 
\begin{equation*}\label{eq:BT-inverse-pressure}
 \alpha_{\boldsymbol\lambda}(s):=\liminf_{n\to\infty}\frac{-\log\varphi^s(A_{\lambda_n})}{n}
\end{equation*}
and
\begin{equation*}\label{eq:BT-critical-exponent}
 s_{\mathrm{BT}} := \inf\{s>0:{P_{\boldsymbol A }(s)}\le\alpha_{\boldsymbol\lambda}(s)\}.
\end{equation*}
\begin{theorem}[B\'ar\'any--Troscheit~\cite{BaranyTroscheit2021}]
\label{thm:barany-troscheit-affine}
Suppose that $\boldsymbol A$ satisfies condition \eqref{nquasi-multiplicativity condition}  and $\max_i\sigma_1(A_i)<1/2$. Then 
\begin{equation*}\label{eq:BT-exact-dimension}
 \dim_{\mathrm H}
 W(\Phi_{\mathbf t},\boldsymbol\lambda) =\min\{d,s_{\mathrm{BT}}\}
\end{equation*}
for Lebesgue-almost all $\boldsymbol t$. 
\end{theorem}
{In fact, the inequality $\dimH W(\Phi_{\mathbf t},\boldsymbol\lambda) \le\min\{d,s_{\mathrm{BT}}\}$ holds for all $\boldsymbol t$.}  If $s_{\mathrm{BT}}>d$, then  {$W(\Phi_{\mathbf t},\boldsymbol\lambda)$} has positive $d$-dimensional Lebesgue measure for almost every  $\boldsymbol t$. 

A further extension is given by the path-dependent shrinking targets of Koivusalo, Liao, Morris and Rams \cite{KoivusaloLiaoRams2022}.  Let
$L:\Sigma\to(\mathbb R_+)^{\mathbb N}:\mathbf i \longmapsto(L_n(\mathbf i ))_{n\geq1}$, where $L_n(\mathbf i )$ depends only on $\mathbf i|_n$. For a finite word $\mathbf i$, define $\underline{L}(\mathbf i):=L_{|\mathbf i|}(\omega)$, for any $\omega\in[\mathbf i]$. Suppose that
$L_n(\mathbf i )\to\infty$ for any $\mathbf i \in\Sigma$ and that there is
$\kappa>0$ such that for all finite words $\mathbf i,\mathbf k$
\[
 \bigl|\underline{L}(\mathbf i\mathbf k)-\underline{L}(\mathbf i)-\underline{L}(\mathbf k)\bigr|\leq\kappa.
\]

 Given $\mathbf j\in\Sigma$,  define
\begin{equation*}\label{eq:KLR-path-dependent-set}
W(\Phi_{\boldsymbol t},\mathbf j,L):=\pi_{\boldsymbol t} \left\{ \mathbf i\in\Sigma: \sigma^n\mathbf i \in[\mathbf j|_{L_n(\mathbf i)}] \ \text{for infinitely many }n \right\}.
\end{equation*} 
If $u$ is not an integer, $\mathbf j|_u$ means $\mathbf j|_{\lfloor u\rfloor}$.
Define
\begin{equation}\label{eq:KLMR-limsup-pressure}
 P_{\mathbf j,L}^*(s):=\limsup_{n\to\infty}\frac{1}{n}\log\sum_{\mathbf i\in\mathcal A^n}\varphi^s\bigl(A_{\mathbf i\,\mathbf j|_{L_n(\mathbf i)}}\bigr).
\end{equation}
When the limsup is a limit, denote it by $P_{\mathbf j,L}(s)$.
\begin{theorem}[Koivusalo--Liao--Morris--Rams {\cite{KoivusaloLiaoRams2022}}]
\label{thm:koivusalo-liao-rams-affine}
Suppose that each $A_i$   ($i\in\mathcal{A}$) is invertible, $\max_{i\in\mathcal A}\sigma_1(A_i)<\frac{1}{2}$ and that $(A_i)_{i\in\mathcal A}$ is completely reducible. Let $L$ satisfy the assumptions above. Let $\mu$ be an ergodic $\sigma$-invariant probability measure on $\Sigma$. Then, for $\mu$-almost every $\mathbf j$, the limits in \eqref{eq:KLMR-limsup-pressure} exists and is independent of
$\mathbf j$. For these $\mathbf j$, for Lebesgue almost all $\boldsymbol t$,
\[
 \dim_{\mathrm H}W(\Phi_{\boldsymbol t},\mathbf j,L)=\min\{d,s_0\},
\]
where $s_0$ is the unique value for which $P_{\mathbf j,L}(s_0)=0$
\end{theorem}

{We now pass from generic affine systems to a fixed class of planar self-affine systems.  Let $M>N\ge2$ and let $K$ be the Bedford--McMullen carpet  associated with a digit set $ \mathcal D\subseteq\{1,\ldots,N\}\times\{1,\ldots,M\}.$ Write $D=\#\mathcal D$, %let $T_a$ denote the number of selected rectangles in the $a$-th column, 
and consider
\[
   T(x,y)=(Nx~(\bmod~{1}),\, My~(\bmod~{1}))
\]
on $K$. Let $\mathcal P_k(\bfx)$ denote the level-$k$ cylinder containing $\bfx\in K$. B\'ar\'any and Rams \cite{BaranyRams2018} considered both cylinder targets and geometric ball targets. For an arbitrary sequence ${(\bfx_n)_{n\in\mathbb{N}}}\subset K$ and a function $f:\mathbb N\to\mathbb N$ satisfying $\lim_{n\to\infty}f(n)/n=c>0$, they found the exact formula of 
\[
\dim_{\mathrm H} \left\{ \bfx\in K:T^n \bfx\in\mathcal P_{f(n)}(\bfx_n)  \text{ for infinitely many }n \right\}.
\]
In particular, the formula is independent of the target centres.  

For geometric balls, let $\mu$ be an ergodic $T$-invariant probability measure with $\operatorname{supp}\mu=K$, and let ${(\bfx_n)_{n\in\mathbb{N}}}$ be identically distributed random variables with common distribution $\mu$, not necessarily independent.  If  $\psi : \mathbb{N} \to \mathbb{R}^+$ satisfy
$$ \lim_{n\to\infty} \frac{-\log \psi(n)}{n\log N}=c>0,$$
then B\'ar\'any and Rams \cite{BaranyRams2018}  obtained the Hausdorff dimension  formula for
\begin{equation*}\label{eq:BM-ball-target}
 \left\{  \bfx\in K:T^n \bfx\in B(\bfx_n,\psi(n))  \text{ for infinitely many }n \right\},
\end{equation*}
which may, however, depend on the common distribution  of the target centres. They also established a deterministic symbolic counterpart of the
ball-target result. For sufficiently large $c$, 
the cylinder target and geometric ball target sets both have Hausdorff dimension $\frac{\log D}{(1+c)\log N}.$ In this range, the dimension is independent of the target centres.}

Wang and Wu \cite{WangWu2021} obtained an exact formula for rectangular
targets in products of homogeneous Cantor sets. For $1\le i\le d$, let $C_i$ be a homogeneous Cantor set invariant under $T_i(x)=b_ix\pmod1$. Put $\delta_i:=\dim_{\mathrm H}C_i,$ $ a_i:=\log b_i$, and let $\mathcal C:=\prod_{i=1}^dC_i$. Write $T=T_1\times\cdots\times T_d.$ 
For $\bfy=(y_{1},\ldots,y_{d})\in \mathcal C$ and non-increasing functions $\psi_i : \mathbb N \to (0,\infty)$, define
\begin{equation*}\label{eq:WW-rectangular-target}
 W(T,\bfy,\boldsymbol\psi):=\left\{  \bfx\in \mathcal C: \|b_i^n x_i-x_{0,i}\| <\psi_i(n) \text{ for every }i \text{ and infinitely many }n\right\}.
\end{equation*}
Let $\mathcal{U}(\boldsymbol\psi)$ be the set of accumulation points of 
\begin{equation}\label{accumulationpoints}
\left\{\left(\frac{- \log \psi_1(n)}{n},\dots, \frac{- \log\psi_d (n)}{n}\right): n \ge 1\right\}.
\end{equation}
For {$\boldsymbol v=(v_1,\dots, v_d)\in \mathcal U$}, put 
$${\mathscr S}(\boldsymbol v):=\{a_i,a_i+v_i:1\le i\le d\}.$$
For $u\in{\mathscr S}(\boldsymbol v)$, define
\[
 \begin{split}
 K_1(u)&:=\{i:a_i\ge u\},\quad K_2(u):=\{i:a_i+v_i\le u\}\setminus K_1(u),\\
 K_3(u)&:=\{1,\ldots,d\}\setminus\bigl(K_1(u)\cup K_2(u)\bigr),
 \end{split}
\]
and
\begin{equation*}\label{eq:WW-dimension-number}
 s(\boldsymbol v) :=\min_{u\in\mathscr{S}(\boldsymbol v)}
 \left\{\sum_{i\in K_1(u)}\delta_i+\sum_{i\in K_2(u)}\delta_i +\frac{\sum_{i\in K_3(u)}a_i\delta_i-\sum_{i\in K_2(u)}v_i\delta_i}{u}\right\}.
\end{equation*}
{We shall encounter several quantities of this form, and in each case the same threefold partition of the index set is used, with defining inequalities and the weights adapted to the context.}
\begin{theorem}[Wang--Wu {\cite{WangWu2021}}]
\label{thm:wang-wu-cantor-targets}
Let $\psi_i : \mathbb N \to (0,\infty)$ be non-increasing for all $i$. If $\mathcal{U}$ is bounded, then
\begin{equation*}\label{eq:WW-general-rates}
 \dim_{\mathrm H}W(T,y,\boldsymbol\psi) = \sup_{\boldsymbol v\in \mathcal U}s(\boldsymbol v).
\end{equation*}
\end{theorem}

Earlier results for generic self-affine systems did not generally require separation, but formulated the problem through symbolic codings and their projections. Baker and Koivusalo \cite{BakerKoivusalo2024} 
studied shrinking target sets for overlapping iterated function systems. They obtained results on Lebesgue measure and a lower bound  on Hausdorff dimension in a special diagonal setting.

{%Jordan and Koivusalo \cite{JordanKoivusalo2025} recently treated geometric targets for a fixed overlapping affine family.} 

{Jordan and Koivusalo~\cite{JordanKoivusalo2025} recently studied
geometric shrinking targets for the planar affine family introduced by
Przytycki and Urbański.}
For $1/2<\lambda<1$, consider
\[
 f_0(x,y)=\left(\lambda x,\frac y2\right),\quad f_1(x,y)= \left(\lambda x+1-\lambda,\frac y2+\frac12\right),
\]
and let $F_\lambda$ be its attractor. $E_{\lambda}:F_\lambda\to F_\lambda$ be the expanding map defined by the inverse branches, with a fixed
choice on points having more than one coding.  For $\bfz\in F_\lambda$ and
$0<\gamma<1$, define
\begin{equation*}\label{eq:PU-shrinking-target}
 W(E_\lambda,\bfz,\gamma):= \left\{\bfx\in F_\lambda: E_\lambda^n \bfx\in Q(\bfz,\gamma^n) \text{ for infinitely many }n\right\},
\end{equation*}
where $Q(\bfz,r)$ denotes the square of radius $r$ centred at $\bfz$. They computed the Hausdorff dimension of  $W(E_{\lambda},\bfz,\lambda)$ in three cases.  First, for  $\bfz\in F_\lambda$ and $\lambda\in(1/2,1/(2\gamma))$ outside an exceptional set of Hausdorff dimension zero,
\begin{equation*}\label{eq:PU-first-regime}
 \dim_{\mathrm H} W(E_\lambda,\bfz,\gamma)=\frac{-\log 2}{\log(\lambda\gamma)}.
\end{equation*}
Second, if $\lambda\in(1/(2\gamma),1)$ lies outside an exceptional set of Hausdorff dimension zero and $\bfz$ satisfying a unique expansion condition, then
\begin{equation*}\label{eq:PU-unique-expansion}
 \dim_{\mathrm H} W(E_\lambda,\bfz,\gamma) = 2+\frac{\log\lambda}{\log 2} -\frac{\log\gamma}{\log(\gamma\lambda)}.
\end{equation*}
Finally, for Lebesgue-almost all $\lambda$ in region of transversality (see Definition 2.3 in \cite{JordanKoivusalo2025}) and for almost every $\bfz$ with respect to the natural measure
on $F_\lambda$,
\begin{equation*}\label{eq:PU-typical-centre}
 \dim_{\mathrm H} W(E_\lambda,\bfz,\gamma) = \frac{2\log 2+\log\lambda}{\log(2/\gamma)}.
\end{equation*}

\subsection{Matrix transformations}

Matrix transformations provide another important non-conformal setting
for shrinking target problems.  We begin with the one-dimensional
$\beta$-transformation and then consider higher-dimensional matrices,
where expansion rates and the lattice distribution of inverse-image
components govern the dimension. We conclude with the partly understood hyperbolic case.

\subsubsection{The one-dimensional case.}
Throughout this section, unless stated otherwise, $\psi:\mathbb N
\to(0,\infty)$ denotes an arbitrary positive function, and no monotonicity is assumed.  {Recall that the lower exponential order at infinity of $\psi$ is given by}
\[
\alpha=\alpha(\psi):=\liminf_{n\to\infty}\frac{-\log\psi(n)}{n}.
\]
For $\beta>1$, the beta-transformation is defined by
\[
 T_\beta x=\beta x~(\bmod\ {1}), \quad x\in[0,1).
\]
 Given  $y\in[0,1)$,  write
\[
 W(T_\beta,y,\psi):=\{x\in[0,1):T_\beta^n x\in B(y,\psi(n))\text{ for infinitely many }n\}.
\]

\begin{theorem}[Shen--Wang {\cite{ShenWang2013}}]
\label{thm:shen-wang-beta}
Let $\psi:\mathbb N\to(0,\infty)$ be a positive function and $y\in [0,1)$. Then for any $\beta>1$,
\begin{equation*}\label{eq:beta-shrinking-matrix}
 \dim_{\mathrm H} W(T_\beta,y,\psi)=\frac{\log\beta}{\alpha+\log\beta}.
\end{equation*}
\end{theorem}

\noindent In a later relevant work, Coons, Hussain and Wang \cite{coonshussain2016} obtain a Hausdorff measure dichotomy law of the set $W(T_{\beta},y,\psi)$.
%In particular, for $r_n=e^{-\alpha n}$,
%\begin{equation}\label{eq:beta-exponential-matrix}
% \dimH W(T_\beta,\boldsymbol r,y) =\frac{\log\beta}{\alpha+\log\beta}.
%\end{equation}
%This is the basic matrix formula when there is only one expanding direction.

\subsubsection{Expanding matrices in higher dimensions.}Let $GL_d(\Z)$ be the collection of all $d\times d$ non-singular integer matrices.  If $A\in {GL_d}(\Z)$, then
\[
 T_A\bfx=A\bfx\pmod{1},\quad \bfx\in\T^d,
\]
defines a toral endomorphism. For a sequence of centres $\boldsymbol y=(\bfy_n)_{n\geq1}\subset\mathbb T^d$, set
\[
   W(T_A,\boldsymbol y,\psi):=\left\{\bfx\in\mathbb T^d:  T_A^n \bfx\in B(\bfy_n,\psi(n))
      \text{ for infinitely many }n\right\}.
\]
When $\bfy_n=\bfy$ for all $n$, we write $W(T_A,\bfy,\psi)$.

For a general real matrix $A$, the map $T_A$ can instead be defined as $\bfx\mapsto A\bfx\pmod1$ on $[0,1)^d$, where reduction modulo one is taken coordinatewise. Unless $A$ is integer, this map does not define a group endomorphism of $\T^d$.  In general,
\[
T_A^{\,n}\bfx\neq A^n \bfx\pmod 1.
\]
Thus, for a non-integer matrix $A$, the shrinking target set defined by the iterates of $T_A$ must be distinguished from the  limsup set defined by matrix sequence $(A^n)_{n\ge1}$. This distinction is important here because this section considers both real matrix transformations and toral endomorphisms.

Assume first that $A$ is expanding, and order its eigenvalues,
counted with algebraic multiplicity, so that
\[
 1<|\lambda_1|\le\cdots\le|\lambda_d|.\]
Put $ \ell_i:=\log|\lambda_i|$, $1\le i\le d$. 

The inverse image of a ball under $T_A^n$ is a union of sets comparable to
ellipsoids, whose principal semi-axes are governed by the expansion rates $\ell_1,\ldots,\ell_d$. %Covering these ellipsoids by balls at each principal scale gives
The standard covering argument gives
\begin{equation}\label{eq:toral-covering-bound}
 \dimH W(T_A,\bfy,\psi)\le D_{\rm cov}(\alpha):=\min_{1\le i\le d} \frac{i\ell_i+\sum_{j>i}\ell_j}{\alpha+\ell_i}.
\end{equation}
If $\psi$ is non-increasing, Hill and Velani \cite{HillVelani1999} proved that this upper bound is attained whenever
\[
 \alpha>\ell_d-\ell_1.
\]
%In this range, the number and sizes of the inverse-image components determine the dimension.  Outside this range, their distribution may also be relevant.
In other words, when the targets shrink faster than the difference between the largest and smallest expansion rates, the Hausdorff dimension is completely determined by the number and sizes of the inverse-image components. For smaller values of $\alpha$, some directions may expand much faster than others. At the covering scale associated with a slower direction, many inverse-image components may therefore be too close. The Hausdorff dimension may then depend not only on the number and shape of these components, but also on how their centres are distributed.

 Li, Liao, Velani and Zorin \cite{LiLiaoVelaniZorin2023} later considered diagonal real matrices with positive entries, and obtained an exact formula for rectangular targets. Let  $A=\operatorname{diag}(\lambda_1,\ldots,\lambda_d)$ with $1<\lambda_1\le\cdots\le\lambda_d$. In this case, $T_A$ can be regarded as a product of beta-transformations
\[
 T_A=T_{\lambda_1}\times\cdots\times T_{\lambda_d}
 \quad\text{on }[0,1)^d.
\]
For
\[
   \boldsymbol\psi=(\psi_1,\ldots,\psi_d),\quad \psi_i:\mathbb N\to(0,\infty),
\]
and $\bfy=(y_1,\ldots,y_d)\in[0,1)^d$, define
\[
   R\bigl(\bfy,\boldsymbol\psi(n)\bigr) :=\left\{ \bfx\in[0,1)^d:|x_i-y_i|\leq\psi_i(n) \text{ for }1\leq i\leq d\right\}
\]
and
\[
   W(T_A,\bfy,\boldsymbol\psi):=\left\{ \bfx\in[0,1)^d:T_A^n \bfx\in R\bigl(\bfy,\boldsymbol\psi(n)\bigr)\text{ for infinitely many }n \right\}.
\]
Recall that {$\mathcal{U}(\boldsymbol\psi)$} denotes the set of accumulation points associated with $\boldsymbol\psi$, as defined in \eqref{accumulationpoints}, {and the threefold partition introduced following \eqref{accumulationpoints}. In the present setting, we partition the index set in the same way, with the comparsion now involving $\ell_k$ and $\ell_i+t_i$. For ${\mathbf t}\in \mathcal{U}(\boldsymbol\psi)$ and $1\le i\le d$, set
\[ \mathcal K_1(i):= \left\{ k:\ell_k>\ell_i+t_i
   \right\},\,\, 
   \mathcal K_2(i):=
   \left\{ k:\ell_k+t_k\leq\ell_i+t_i
   \right\},\]
\[\mathcal K_3(i) := \{1,\ldots,d\}\setminus \bigl( \mathcal K_1(i)\cup\mathcal K_2(i)\bigr),\]
and
\[   \theta_i(\mathbf t):=\sum_{k\in\mathcal K_1(i)}1 +\sum_{k\in\mathcal K_2(i)}\left(1-\frac{t_k}{\ell_i+t_i}\right) +
   \sum_{k\in\mathcal K_3(i)}\frac{\ell_k}{\ell_i+t_i}.\]
}

\begin{theorem}[Li--Liao--Velani--Zorin {\cite{LiLiaoVelaniZorin2023}}]\label{thm:llvz-rectangular-targets} 
Let $A$ satisfy the assumptions above. Suppose that $\mathcal U(\boldsymbol\psi)$ is bounded.  Let $\bfy\in[0,1)^d$. Then
\begin{equation}\label{eq:toral-eigenvalue-formula-M}
   \dim_{\mathrm H}W(T_A,\bfy,\boldsymbol\psi)
   =
   \sup_{\mathbf t\in{\mathcal{U}(\boldsymbol\psi)}}
   \min_{1\leq i\leq d}\theta_i(\mathbf t).
\end{equation}
\end{theorem}
When $\psi_1=\cdots=\psi_d=\psi$, if $\alpha<\infty$, then Theorem~\ref{thm:llvz-rectangular-targets} reduces to
\begin{equation*}
\label{eq:toral-eigenvalue-formula}
   \dim_{\mathrm H}W(T_A,\bfy,\psi)= D_{\mathrm{eig}}(\alpha),
\end{equation*}
where
\begin{equation}\label{eq:toral-eigenvalue-formula}D_{\rm eig}(\alpha):=\min_{1\le i\le d}\frac{ i\ell_i+\sum_{j>i}\ell_j -\sum_{\ell_j>\ell_i+\alpha} (\ell_j-\ell_i-\alpha)}{\alpha+\ell_i}.
 \end{equation}
If $\alpha=\infty$, then $\dim_{\mathrm H}W(T_A,\bfy,\psi)=0.$
They also obtained the same expression \eqref{eq:toral-eigenvalue-formula} for expanding integer matrices diagonalizable over $\Z$  and have
positive eigenvalues when $\psi_1=\cdots=\psi_d=\psi$. 
 Later, He and Liao \cite{HeLiao2026} pointed out that \eqref{eq:toral-eigenvalue-formula-M} holds for all expanding diagonal real matrices, under the assumption that each $\psi_i$ is non-increasing.

%The same line of research also considers targets other than balls.  Li, Liao, Velani and Zorin \cite{LiLiaoVelaniZorin2023} studied rectangular and multiplicative targets, while

Further extensions concern other target sets and orbit-dependent shrinking rates.  Li, Liao, Velani and Zorin \cite{LiLiaoVelaniZorin2023} studied multiplicative targets, while in two separate works, Hussain and Wang \cite{hussain2018,HussainWang2023}  studied coordinate-dependent shrinking rates determined, respectively, by the usual sequence and by ergodic sums in product $\beta$-transformations. He \cite{HeHausdorffMeasure2025} established Hausdorff measure dichotomies for rectangular and multiplicative targets whose centres may depend Lipschitz-continuously on the initial point. The divergence result for rectangular targets, the parameters are assumed to be integers, or more generally to satisfy the corresponding Markov condition. 
He \cite{He2025Parallelepiped} further considered arbitrary parallelepiped targets for product $\beta$-transformations.  In this setting, the orientation of the target can affect the dimension even when its side lengths are fixed.  % After using Gram--Schmidt orthogonalisation, he obtained the exact Hausdorff dimension and a large-intersection result.  This shows that, in addition to the expansion rates of the matrix, the geometry of the target may enter the dimension formula.

%The formula $D_{\rm eig}(\alpha)$ is not valid for every expanding integer matrix.
For general expanding integer matrices, if $\psi$ is non-increasing, Hu, Persson, Wu and Zhang \cite{HuPerssonWuZhang2025} proved that for any  $\boldsymbol y=(\bfy_n)_{n\geq1}$,
\begin{equation}\label{eq:HPWZ-bounds}
 D_{\rm eig}(\alpha) \leq \dimH W(T_A,\boldsymbol y,\psi) \leq D_{\rm cov}(\alpha).
\end{equation} 
The gap in \eqref{eq:HPWZ-bounds} exists. Hu, Huang, Li and Wu \cite{HuHuangLiWu2026} constructed an example for which
\[
 D_{\rm eig}(\alpha) < \dimH W(T_A,\bfy,\psi)< D_{\rm cov}(\alpha),
\]
and it shows that neither $D_{\rm eig}(\alpha)$ nor $D_{\rm cov}(\alpha)$ determines the dimension in general.

Hu, Huang, Li and Wu \cite{HuHuangLiWu2026} gave a more general treatment, incorporating the distribution of the inverse-image components, in the setting of matrix sequences.  Let $\mathcal A=(A_n)_{n\geq1}$, with $ A_n\in GL_d(\mathbb R)$, and define
\begin{equation*}\label{eq:matrix-sequence-W-shrinking}
W(\mathcal A,\bfy,\psi): =\{\bfx\in[0,1)^d:A_n\bfx~(\bmod\,1)\in B(\bfy,\psi(n))  \text{ for infinitely many }n\}.
\end{equation*}
If $A\in GL_d(\mathbb Z)$ and $A_n=A^n$, this is exactly $W(T_A, \bfy,\psi)$. For a lattice $\Lambda\subset\mathbb R^d$, let
\[
 m_j(\Lambda) :=\inf\left\{r>0:   \dim\operatorname{span}\bigl(\Lambda\cap B(0,r)\bigr)\ge j  \right\}, \quad 1\le j\le d,
\]
be its successive minima.  
\begin{theorem}[A simplified version of {\cite{HuHuangLiWu2026}}]
\label{thm:hhlw-matrix-sequences}
Let $A\in GL_d(\mathbb Z)$ with eigenvalues $1<|\lambda_1|\le\cdots\le|\lambda_d|$, and let $\psi:\mathbb N\to(0,\infty)$ be non-increasing. Then for any $\bfy\in[0,1)^d$,
\begin{equation}\label{eq:toral-lattice-bounds}
 D_{\rm eig}(\alpha)\le \dimH W(T_A, \bfy,\psi) \le D_{\rm lat}(\alpha),
\end{equation}
where
\begin{equation}\label{eq:toral-lattice-upper}
 D_{\rm lat}(\alpha): =\min_{1\le i\le d}
 \frac{ i\ell_i+\sum_{j>i}\ell_j -\sum_{h_j>\ell_i+\alpha}(h_j-\ell_i-\alpha)}
 {\alpha+\ell_i},
\end{equation}
and
\begin{equation*}\label{eq:successive-minima-exponents}
 h_j:=\liminf_{n\to\infty}
      \frac{-\log m_j(A^{-n}\mathbb Z^d)}{n}.
\end{equation*}
\end{theorem}
 Since the last sum in \eqref{eq:toral-lattice-upper} is non-negative,
\[
 D_{\rm lat}(\alpha)\le D_{\rm cov}(\alpha).
\]
The examples in \cite{HuHuangLiWu2026} show that the inequality can be strict.
Thus $D_{\rm lat}$ refines the classical covering upper bound by taking into account the successive minima of $A^{-n}\mathbb Z^d$.

The same work establishes exact formulae for several classes of matrices.  If $A$ is an expanding integer matrix diagonalizable over $\mathbb Q$, then the bounds in \eqref{eq:toral-lattice-bounds} agree and
\[
 \dimH W(T_A, \bfy,\psi) =D_{\rm eig}(\alpha).
\]
The non-diagonalizable case is also treated for expanding Jordan matrices. Let
\[
 A=\operatorname{diag}\bigl(
 J(\lambda_1,n_1),\ldots,J(\lambda_s,n_s)
 \bigr),
 \qquad |\lambda_j|>1.
\]
For $1\le i\le s$, 
\[
 \begin{split}
 \mathcal{K}_1(i)&=\{j:\ell_j>\ell_i+\alpha\},\quad  \mathcal{K}_2(i)=\{j:\ell_j<\ell_i\},\\
 \mathcal{K}_3(i)&=\{1,\ldots,s\}\setminus
         \bigl(\mathcal{K}_1(i)\cup \mathcal{K}_2(i)\bigr).
 \end{split}
\]
Let $\cA=\{A^n\}_{n\in\N}$. Then
\begin{equation}\label{eq:jordan-matrix-formula}
 \dimH W(\cA,\bfy,\psi)
 =\min_{1\le i\le s}
 \left\{
  \sum_{j\in \mathcal{K}_1(i)}n_j
  +\sum_{j\in \mathcal{K}_2(i)}
       \frac{n_j\ell_i}{\ell_i+\alpha}
  +\sum_{j\in \mathcal{K}_3(i)}
       \frac{n_j\ell_j}{\ell_i+\alpha}
 \right\}.
\end{equation}
By rational conjugacy, the same formula applies to the corresponding expanding integer matrices with such a rational Jordan decomposition.

\subsubsection{Hyperbolic integer matrices.}
 Let $A\in GL_2(\mathbb Z)$, and suppose that its eigenvalues satisfy
\[
 0<|\lambda_1|<1<|\lambda_2|. 
\]
Write $\ell_i=\log|\lambda_i|.$

\begin{theorem}[Hu--Persson--Wu--Zhang {\cite{HuPerssonWuZhang2025}}]
\label{thm:hpwz-hyperbolic-matrices}
Let $\psi:\mathbb N\to(0,\infty)$ be non-increasing. Let $\boldsymbol y=(\bfy_n)_{n\geq1}\subset\mathbb T^2$. Then
\begin{equation}
\label{eq:hyperbolic-W}
   \dim_{\mathrm H}W(T_A,\boldsymbol y,\psi)=
   \begin{cases}
      \displaystyle
      \frac{2\ell_2}{\alpha+\ell_2},
      & 0<\alpha<-\ell_1,\\[3ex]
      \displaystyle
      \min\left\{
         \frac{\ell_1+\ell_2}{\alpha+\ell_1},
         \frac{2\ell_2}{\alpha+\ell_2}
      \right\},
      & \alpha>-\ell_1.
   \end{cases}
\end{equation}
\end{theorem}
At the critical value $\alpha=-\ell_1$, the answer depends on
the determinant. If $|\det A|>1$, then
\[
 \dimH W(T_A,\boldsymbol y,\psi) =\frac{2\ell_2}{\ell_2-\ell_1}.
\]
If $A$ is an automorphism, then $\ell_1+\ell_2=0$, and the dimension is zero for $\alpha>-\ell_1$. When $\alpha=-\ell_1$, the limsup set may be empty or may
contain a line segment, depending on the sequence of centres.   
No general exact formula is currently known for hyperbolic integer matrices in dimension at least three.

\section{{Recurrent} sets}\label{sec:recurrence}

This section surveys Hausdorff measure and dimension results for  {recurrent} sets. We first consider expanding Markov maps and conformal iterated function systems, and then turn to affine iterated function systems. Finally, we discuss {recurrent} sets for  expanding and hyperbolic matrix transformations.

\subsection{Conformal systems}

\subsubsection{Expanding Markov maps}
\label{subsec:expanding-Markov-recurrence}

The class of expanding Markov maps considered here differs from the
one introduced in Section~3.1. We retain the notation for branches,
cylinders, and repellers, but not the standing assumptions imposed
there. In Section~3.1 each branch was assumed to have the full-image
property. By contrast, the finite-branch systems considered by Yuan
and Li~\cite{YuanLiLipschitz2023} need not be full branch, but satisfy additional  transitivity, and weak \(C^{1+\delta}\) conditions. For countably many branches, their definition additionally requires the full-image property. Throughout this subsection,  let 
$$T:X=\overline{\bigcup_{i\in\mathcal A}I_i}\to[0,1],$$
be an expanding Markov map in the sense of
\cite[Definition~1.1]{YuanLiLipschitz2023}, where  $\mathcal A$ may be finite or countably infinite with $\#\mathcal{A}\ge 2$ and $\{I_i\}_{i\in \mathcal{A}}$ are disjoint open sub-intervals of $[0,1)$. Let $\Sigma_n\subseteq\mathcal A^n$ denote the set of admissible
words of length $n$, let $I_\omega$ be the cylinder corresponding
to $\omega\in\Sigma_n$, and let  {$\Lambda$} be the repeller of $T$.

Let  $\psi:\mathbb N\to(0,\infty)$, and let $\alpha=\alpha(\psi)$ be its lower {exponential} order at infinity.  Define the {recurrent} set
\begin{equation}\label{eq:one-dimensional-recurrence}
 R(T,\psi):=\left\{ x\in {\Lambda} : T^n x\in B(x,\psi(n))\text{ for infinitely many }n\right\}.
\end{equation}
For a continuous function $f:X\to(0,\infty)$, {let $\psi_f(n,x)=\exp\left(-S_nf(x)\right)$. Define
\[
   R(T, \psi_f):= \left\{ x\in\Lambda: T^n x\in B\bigl(x,\psi_f(n,x)\bigr) \ \text{for infinitely many }n \right\}.
\]}

{The pressure used below is the same as that introduced in \eqref{eq:markov-pressure}. For convenience, we recall its explicit form. } For a  potential $g:[0,1)\to\mathbb R$,
 \[
   P(g) := \liminf_{n\to\infty}\frac{1}{n} \log
   \sum_{\omega\in\Sigma^n}\exp\left(\sup_{x\in I_\omega}S_ng(x)\right).
\]
%If $\mathcal A$ is finite, this agrees with the topological pressure introduced in Section~\ref{sec:preliminaries}.
 
 Yuan and Li \cite{YuanLiLipschitz2023} studied the more general sets obtained by replacing the centre $x$ in the target by a sequence of uniformly Lipschitz functions.  Specializing to the identity target centre
gives the following recurrence results.
\begin{theorem}[Yuan--Li {\cite{YuanLiLipschitz2023}}] \label{thm:Yuan-Li-Markov-recurrence}
Let $T$, $f$, and $\psi$ be as above. 
\begin{enumerate}
\item If $\mathcal A$ is finite, then 
\begin{equation*}\label{eq:finite-Markov-Birkhoff-recurrence}
 \dim_{\mathrm H}{R(T, \psi_f)} =s(f),
\end{equation*}
where $s(f)$ is the unique solution of
\[
 P\!\left(-s(f+\log|T'|)\right)=0.
\]
\item If $\mathcal A$ is countably infinite, then
\[\dim_{\mathrm H}{R(T, \psi_f)}=\inf\left\{s\ge0:P\!\left(-s(f+\log|T'|)\right)\le0\right\}.\]

\item  Suppose that $\mathcal A$ is finite.  If $1\le\alpha(\psi)<\infty$, then
\begin{equation*}\label{eq:finite-Markov-recurrence}
 \dim_{\mathrm H}R(T,\psi)=s_\psi,
\end{equation*}
where $s_\psi$ is the unique solution of
\begin{equation*}\label{eq:finite-Markov-recurrence-pressure}
 P\left(-s\bigl(\log|T'|+\log\alpha(\psi)\bigr)\right)=0.
\end{equation*}
If $\alpha(\psi)=\infty$, then $\dim_{\mathrm H}R(T,\psi)=0.$ 

\item Suppose  that $\mathcal A$ is countably infinite.  If $\psi$ is decreasing, $1\le\alpha(\psi)<\infty$, and ${\overline \Lambda}=[0,1]$. Then
\begin{equation*}\label{eq:countable-Markov-recurrence}
 \dim_{\mathrm H}R(T,\psi)=\inf\left\{s\ge0:P\left(-s\bigl(\log|T'|+\log\alpha(\psi)\bigr)\right)\le0\right\}.
\end{equation*}
\end{enumerate}
\end{theorem}
In the last statement, the monotonicity assumption on $\psi$ may be removed if  $\inf_{x\in\bigcup_{i\in\mathcal A}I_i}|T'(x)|>1.$
For countably many branches, no general formula is obtained in
\cite{YuanLiLipschitz2023} when $\alpha(\psi)=\infty$.

Peng, Tan and Wang \cite{PengTanWang2012} also studied quantitative recurrence properties for the Gauss map through the first return time to a ball centered at the initial point.  They determined the Hausdorff dimensions of the level sets of points with prescribed recurrence time.

\subsubsection{Conformal iterated function systems}
\label{sec:recurrence-conformalifs}
We retain the notation for $C^{1+\varepsilon}$ conformal
iterated function systems introduced in Section~3.1.

 Seuret and Wang~\cite{SeuretWang2015} considered a slightly broader
class of systems. Let
\[
 \Phi=\{\phi_i:i\in\mathcal A\}
\]
be a finite or countable conformal iterated function system on $[0,1]^d$ satisfying the open set condition, and let $J$ be its attractor. The system is assumed to be uniformly contracting, and its
maps extend to $C^1$ conformal diffeomorphisms on a common
neighbourhood of $[0,1]^d$ and satisfy the bounded distortion
property uniformly over all finite words.

Let $\pi:\mathcal A^{\mathbb N}\to J$ be the coding map. For $\mathbf i=(i_1,\ldots,i_n)\in{\mathcal A}^n$, write 
\[\phi_{\mathbf i}:=\phi_{i_1}\circ\cdots\circ\phi_{i_n},\quad J_n(\mathbf i):=\phi_{\mathbf i}([0,1]^d).\]
For $\bfx\in J_n({\mathbf i})$, set
\[
   S_n^{\mathbf i}f(\bfx):=
   f(\bfx)+f\!\left(\phi_{i_1}^{-1}(\bfx)\right)
   +\cdots+
   f\!\left(\phi_{i_1\cdots i_{n-1}}^{-1}(\bfx)\right).
\]
{Assume that $f:[0,1]^d\to(0,\infty)$ satisfies }
\begin{equation}\label{vanishing-variation}
   \operatorname{Var}_1(f)<\infty   \quad\text{and}\quad\operatorname{Var}_n(f)\to 0,
\end{equation}
where $\operatorname{Var}_n(f):=\sup_{\mathbf i\in\mathcal A^n}\sup_{\bfx,\bfy\in J_n({\mathbf i})}|f(\bfx)-f(\bfy)|.$ 
{ 
In some of the literature, condition \eqref{vanishing-variation} is also referred to as the tempered distortion property. It is stronger than the condition adopted in Section 3.1.2.} For such $f$, define
\begin{equation}\label{eq:seuret-wang-recurrence}
   R_\Phi(f):= 
   \left\{
      \bfx\in J:
      \begin{array}{l}
      \text{there exists a coding \(\boldsymbol\omega\in\pi^{-1}(\bfx)\) such that}\\
      \displaystyle
      \left|\bfx-\phi_{\boldsymbol\omega|_n}^{-1}(\bfx)\right|
      <\exp\!\left(-S_n^{\boldsymbol\omega|_n}f(\bfx)\right)
      \text{ for infinitely many }n
      \end{array}
   \right\}.
\end{equation}
Thus, if $\bfx$ has more than one coding, it is enough that the recurrence condition hold along one of them.

For \(t\geq0\), set
\[
   P_\Phi(t;f):=
   \lim_{n\to\infty}\frac1n
   \log\sum_{\mathbf w\in\mathcal A^n}
   \sup_{\bfx\in J_n(\mathbf w)}
   \left(\left|(\phi_{\mathbf w}^{-1})'(\bfx)\right|^{-1}
      \exp\!\left(-S_n^{\mathbf w}f(\bfx)\right)
   \right)^t
\]

\begin{theorem}[Seuret and Wang {\cite{SeuretWang2015}}]
\label{thm:Seuret-Wang-conformal-recurrence}
Let $\Phi$ be a  conformal iterated function system
on $[0,1]^d$ as above. Let $f:[0,1]^d\to(0,\infty)$ satisfy \eqref{vanishing-variation}. Then
\begin{equation*}\label{eq:seuret-wang-dimension}
 \dim_{\mathrm H} R_\Phi(f)=t(f):=\inf\{t\geq0:P_\Phi(t;f)\leq0\}.
\end{equation*}
\end{theorem}

The recurrence radius in this result has the ergodic sum form
appearing in \eqref{eq:seuret-wang-recurrence}. In particular,
$f\equiv\alpha$ gives the exponential radius $e^{-\alpha n}$,
whereas an arbitrary numerical sequence $\psi(n)$ is not covered directly.

The Hausdorff measure problem for {recurrent} sets defined by arbitrary numerical sequences was studied later in the homogeneous similar setting by Chang, Wu and Wu~\cite{ChangWuWu2019}.  Let
\[
 \Phi_\rho = \{\phi_j(x)=\rho x+a_j:1\leq j\leq L\}, \quad 0<\rho<1,
\]
be a homogeneous similar iterated function system on $[0,1]$ satisfying the strong separation condition, and let $K$ be its attractor. Put $\gamma:=\dim_{\mathrm H}K .$  The coding map $\pi_\rho$  is then one-to-one, and the left shift induces $T_\rho : K \to K$ as follows
\[T_\rho(\pi_\rho(\omega))=\pi_\rho(\sigma\omega)).\]
For $\psi:\mathbb N\to(0,\infty)$, define
\[
 R(T_\rho,\psi):=\left\{x\in K: |T_\rho^nx-x|<\psi(n) \text{ for infinitely many }n\right\}.
\]

\begin{theorem}[Chang, Wu and Wu {\cite{ChangWuWu2019}}]
\label{thm:Chang-Wu-Wu-recurrence}
Let $g$ be a doubling dimension function such that $r^{-\gamma}g(r)$ is increasing as $r\to0$.  Then
\begin{equation}\label{eq:homogeneous-general-gauge}
 \mathcal H^g( R(T_\rho,\psi))=\begin{cases}
 0,&\displaystyle\sum_{n=1}^{\infty}\rho^{-\gamma n}g\!\left(\rho^n\psi(n)\right)<\infty,\\[3mm]
 \mathcal H^g(K),&\displaystyle\sum_{n=1}^{\infty}\rho^{-\gamma n}g\!\left(\rho^n\psi(n)\right)=\infty.
 \end{cases}
\end{equation}
\end{theorem}
Theorem~\ref{thm:Chang-Wu-Wu-recurrence} gives 
\begin{equation*}\label{eq:homogeneous-recurrence-dimension}
   \dim_{\mathrm H}R(T_\rho,\psi)=\frac{\gamma}{1-\alpha(\psi)/\log\rho}.
\end{equation*}
where \(\alpha(\psi)\) is the lower exponential order defined
previously. If we take $g(r)=r^\gamma$, the conclusion becomes
a zero--full law for the finite measure $\mathcal H^\gamma|_K$,.

Baker and Farmer~\cite{BakerFarmer2021} extended this result on 
natural-dimensional Hausdorff measure to finite conformal iterated function
systems. Let $\Phi$ be a finite conformal iterated function system as in Section~3.1.3, with attractor $X$. If $\delta_\Phi=\dim_{\mathrm H}X,$
then
\[
   \sum_{n=1}^{\infty}\psi(n)^{\delta_\Phi}<\infty
   \quad\Longrightarrow\quad
   \mathcal H^{\delta_\Phi}\bigl(R(T_\Phi,\psi)\bigr)=0.
\]
Under the open set condition, if $\sum_{n=1}^{\infty}\psi(n)^{\delta_\Phi}=\infty$, then
\[
   \mathcal H^{\delta_\Phi}\bigl(R(T_\Phi,\psi)\bigr)
   =
   \mathcal H^{\delta_\Phi}(X).
\]
Neither this result nor that of \cite{ChangWuWu2019} requires $\psi$ to be monotone. In a forthcoming paper, Chang, Hussain, Li and Simmons \cite{changhausdorff} develop a Hausdorff theory of recurrence in a broad subshift setup, which is applicable to $\beta$-transformations, conformal repeller, and conformal IFS.

%\subsection{Irrational rotations}
%For the irrational rotation
%\[ T_\alpha x=x+\alpha\pmod1, \quad \alpha\notin\mathbb Q,\]
%the recurrence set has a particularly simple form.  Indeed,
%\[ d(T_\alpha^n x,x)=\|n\alpha\|\]
%is independent of $x$, and hence
%\[R_\alpha(\psi)=\begin{cases}\mathbb T,&\|n\alpha\|<\psi(n)  \ \text{for infinitely many }n,\\ \varnothing,&\text{otherwise}.
% \end{cases}\]
%Thus its Hausdorff dimension is either $1$ or $0$. Therefore fixed irrational rotations do not give a non-trivial Hausdorff-dimension theory for recurrence sets.

\subsection{Affine iterated function systems}

{Recurrent} sets for affine systems have been studied in both symbolic and geometric settings, according to whether returns are measured in the coding space or directly on the attractor.

We first consider projected symbolic {recurrent} sets. Retain the notation
$\Phi_{\boldsymbol t}{=\{f_i:i\in\mathcal{A}\}}$, $K_{\boldsymbol t}$, $\Sigma$, $\sigma$,
$\pi_{\boldsymbol t}$, {$\boldsymbol A$}, $A_{\mathbf i}$, $\varphi^s$, and ${P_{\boldsymbol A}}$ introduced in Section~3.3. Define %the square-pressure function by
\begin{equation*}\label{eq:self-affine-square-pressure}
 P_{{\boldsymbol A,2}}(s)=\lim_{n\to\infty} -\frac{1}{n} \log \sum_{\mathbf i\in\mathcal A^n} \varphi^s(A_{\mathbf i})^2.
\end{equation*}
The limit exists by the submultiplicativity of the singular-value
function. 

Let $\ell=(\ell_n)_{n\geq1}$ be a sequence of positive integers. B\'ar\'any and Troscheit \cite{BaranyTroscheit2021} studied the following sets
\begin{equation}\label{eq:self-affine-cylindrical-recurrence}
R(\Phi_{\boldsymbol t},\ell):=\pi_{\boldsymbol t}\left\{\boldsymbol\omega\in\Sigma:\sigma^n\boldsymbol\omega\in
 [\boldsymbol\omega|_{\ell_n}]\ \text{for infinitely many }n\right\},
\end{equation}
Thus, the target at time $n$ is determined by the initial $\ell_n$ symbols of the orbit itself. Set
\[
 \eta:=\liminf_{n\to\infty}\frac{\ell_n}{n}.
\]

\begin{theorem}[B\'ar\'any--Troscheit
\cite{BaranyTroscheit2021}]
\label{thm:barany-troscheit-recurrence}
Suppose that Condition \eqref{nquasi-multiplicativity condition} holds, $\max_{i\in\mathcal A}\sigma_1(A_i)<\frac{1}{2}$, and $\beta<1$. Let $s_0$ be the unique solution of
\begin{equation*}\label{eq:self-affine-recurrence-pressure}
 (1- \eta)P_{{\boldsymbol A}}(s_0) =  \eta P_{{\boldsymbol A,2}}(s_0).
\end{equation*}
Then, for Lebesgue-almost every $\boldsymbol t$,
\begin{equation*}\label{eq:barany-troscheit-recurrence}
 \dim_{\mathrm H}
 R(\Phi_{\boldsymbol t},\ell) = \min\{d,s_0\}.
\end{equation*}
Moreover, if $s_0>d$, then $\mathcal L^d \bigl(R(\Phi_{\boldsymbol t},\ell)\bigr)>0$ for Lebesgue-almost every $\boldsymbol t$.
\end{theorem}

The recurrence condition in
\eqref{eq:self-affine-cylindrical-recurrence} is imposed in the symbolic space before projection. Consequently, Theorem
\ref{thm:barany-troscheit-recurrence} does not directly give the
dimension of {recurrent} sets defined using the distance on the attractor.

We next turn to geometric {recurrent} sets on the Bedford--McMullen carpet
$K$ introduced in Section~3.3. Recall that 
\[T(x,y)=(T_N(x),T_M(y))=\left(Nx~(\bmod{1}),My~(\bmod{1})\right),
 \quad M>N\geq2.\]
For a positive function $\psi:\mathbb N\to(0,\infty)$ and
$\gamma\geq0$, define
\begin{equation}\label{eq:wu-yuan-recurrence}
 R_\gamma(T,\psi): = \left\{ (x,y)\in K:
 \begin{array}{l}
 T_N^n x\in B(x,\psi(n)),\\
 T_M^n y\in B(y,\psi(n)^\gamma)
 \end{array}
 \text{ for infinitely many }n
 \right\},
\end{equation}
Let $\alpha $ be the lower exponential order at infinity of the function $\psi$.
Set
\[
 \tau_1:=\liminf_{n\to\infty}\frac{-\log_N\psi(n)}{n}=\frac{\alpha}{\log N},\quad
 \tau_2:=\liminf_{n\to\infty}\frac{-\log_M\bigl(\psi(n)^\gamma\bigr)}{n}=\frac{\gamma\alpha}{\log M}.
\]
%In particular, $\tau_2=\frac{\gamma\log N}{\log M}\,\tau_1.$

Wu and Yuan~\cite{WuYuan2025} obtained an exact piecewise formula for
$\dim_{\mathrm H}R_\gamma(T,\psi)$ when
\[
 \tau_1\geq0,
 \qquad
 0\leq\gamma\leq\log_N M,
\]
and every non-empty column of $K$ contains the same number of selected
rectangles. The latter condition is equivalent to $\dim_{\mathrm H}K=\dim_{\mathrm B}K.$ No monotonicity assumption on $\psi$ is required.

 For completeness, consider the usual case $\gamma=1$. Let $C$ be the number of non-empty columns and let $q$ be the common number of selected rectangles in each such column. If $\tau_1\ge 0$, then the formula of \cite{WuYuan2025} becomes 
\[
 \dim_{\mathrm H}R_1(T,\psi)=
 \begin{cases}
 \displaystyle
 \min\left\{\frac{\log_N C+\log_M q}{1+\tau_2},\frac{\log_N C}{1+\tau_1}+\log_M q
 \right\},
 & \log_N M>1+\tau_1,\\[4mm]
 \displaystyle
 \min\left\{\frac{\log_N C+\log_M q}{1+\tau_2},\frac{\log_N C+\log_N q}{1+\tau_1}
 \right\},
 & \log_N M\leq1+\tau_1.
 \end{cases}
\]
Here $R_1(T,\psi)$ is defined using the maximum norm. Its Hausdorff
dimension is unchanged if the maximum norm is replaced by any
equivalent norm.

For overlapping iterated function systems, the inverse dynamics may
be multivalued. Baker and Koivusalo \cite{BakerKoivusalo2024} defined
geometric recurrence by requiring that, for infinitely many $n$, at
least one inverse branch of length $n$ satisfies the return condition.
 %Let
%\[\mathcal S=\{S_i(x)=A_i x+t_i:i\in\mathcal A\}\]
%be a finite affine iterated function system on $\mathbb R^d$ with withnonsingular contracting linear parts and attractor $X$, and set
%\[ T_i=S_i^{-1},\quad\Lambda(\mathcal S)=\sum_{i\in\mathcal A}|\det A_i|.\]
For each $i\in\mathcal A$, let
\[
 T_{i}:=f_{i}^{-1}, \quad {D(\boldsymbol A)}:=\sum_{i\in\mathcal A}|\det A_i|.
\]
For a word $\mathbf i=(i_1,\ldots,i_n)\in\mathcal A^n$, write 
\[T_{\mathbf i}:=T_{i_n}\circ\cdots\circ T_{i_1}.\]
Given $h:\mathbb N\to [0,\infty),$ define
\begin{equation*}\label{eq:overlapping-IFS-recurrence}
 R(\Phi_{\boldsymbol t},h):=\left\{ x\in {K_{\boldsymbol t}}: T_{\mathbf i}x\in B\left(x,\left(h(n){D(\boldsymbol A)}^{-n}\right)^{1/d}\right) \text{ for some }\mathbf i\in\mathcal A^n \text{ and infinitely many }n \right\}.
\end{equation*}
Their general results mainly concern the Lebesgue measure of $  R(\Phi_{\boldsymbol t},h)$, together with its Hausdorff-dimension bounds in certain zero-measure cases. As a representative exact result, consider 
\[
 \Phi(\lambda)=\{S_0(x)=\lambda x,\ S_1(x)=\lambda x+1\}, \quad \frac{1}{2}<\lambda<1,
\]
whose attractor is
\[
 K(\lambda)=\left[0,\frac1{1-\lambda}\right].
\]
Suppose that $\lambda^{-1}$ is a Garsia number {(see Definition 1.1 in \cite{HarePanju2013})}, and $h$ is bounded, then, for any ball $B\subset  K(\lambda)$ and any $0<s\le1$,
\begin{equation*}\label{eq:baker-koivusalo-recurrence-measure}
 \mathcal H^s\bigl(B\cap R(\Phi(\lambda),h)\bigr)=
 \begin{cases}
0,&\displaystyle   \sum_{n=1}^{\infty} 2^{n(1-s)}h(n)^s<\infty,\\[3mm]
\mathcal H^s(B),&\displaystyle \sum_{n=1}^{\infty}2^{n(1-s)}h(n)^s=\infty.
 \end{cases}
\end{equation*}

\subsection{Matrix transformations}

\subsubsection{The one-dimensional case.}

Retain the notation $T_\beta$ and $\alpha=\alpha(\psi)$ introduced in
Section~3.4.1. For $\beta>1$ and a positive function $\psi:\mathbb N\to(0,\infty)$ with $\psi(n)\to0$, define
\[
R(T_\beta,\psi) :=\{x\in[0,1):T_\beta^n x\in B(x,\psi(n))  \text{ for infinitely many }n\}.
\]

\begin{theorem}[Tan--Wang \cite{TanWang2011}]
\label{thm:tan-wang-beta-recurrence}
%Let $\psi$ be a positive function defined on $\mathbb{N}$. 
For any $\beta > 1$,
\begin{equation}\label{eq:beta-recurrence-matrix}
 \dim_{\mathrm H}R(T_\beta,\psi) = \frac{\log\beta}{\alpha+\log\beta}.
\end{equation}
\end{theorem}
No monotonicity assumption on $\psi$ is required.
Comparison with the result on shrinking target sets
of Shen and Wang {\color{brown}\cite{ShenWang2013}} in Section~3.4.1 gives, for every $y\in[0,1)$,
\[
 \dim_{\mathrm H}R(T_\beta,\psi)
 =
 \dim_{\mathrm H}W(T_\beta,y,\psi).
\]
Thus, for $\beta$-transformations, {recurrent} sets and shrinking
target sets have the same Hausdorff dimension.

\subsubsection{Expanding matrices in higher dimensions.}
We first consider diagonal real matrix transformations. Let $\boldsymbol\beta={\rm diag}\{\beta_1,\ldots,\beta_d\}$ with $1<|\beta_1|\le \cdots\le |\beta_d|$, and define
\[
% T_{\boldsymbol\beta} =T_{\beta_1}\times\cdots\times T_{\beta_d} \quad\text{on }[0,1)^d.
{T_{\boldsymbol\beta}\bfx=\boldsymbol\beta\bfx~(\bmod~1)}.
\]
For $\boldsymbol\psi=(\psi_1,\ldots,\psi_d)$ with $\psi_i:\mathbb N\to(0,\infty)$, set
{\[  R(\bfx,\boldsymbol\psi):=\{\bfy\in[0,1)^d:|y_i-x_i|<\psi_i(n)\text{ for }1\le i\le d\},\]
and
\[
R(T_{\boldsymbol\beta},\boldsymbol\psi) =\left\{\bfx\in[0,1)^d:   T_{\boldsymbol\beta}\bfx\in  R(\bfx,\boldsymbol\psi)\text{ for infinitely many }n \right\}.
\]}
{Put $\ell_i=\log|\beta_i|$, $1\le i\le d$. Recall that $\mathcal U(\boldsymbol\psi)$ denotes the set of accumulation points of
\[
 \left(
  \frac{-\log\psi_1(n)}{n},\ldots,\frac{-\log\psi_d(n)}{n}
 \right)_{n\ge1}.
\]
For $\boldsymbol t=(t_1,\ldots,t_d)\in\mathcal U(\Psi)$ and $1\le i\le d$,
\[
 \begin{split}
 \mathcal K_1(i)&=\{k:\ell_k>\ell_i+t_i\},\quad\mathcal K_2(i)=\{k:\ell_k+t_k\le\ell_i+t_i\},\\
 \mathcal K_3(i)&=\{1,\ldots,d\}\setminus\bigl(\mathcal K_1(i)\cup \mathcal K_2(i)\bigr),
 \end{split}
\]
and
\[
 \theta_i(\boldsymbol t) =\sum_{k\in\mathcal K_1(i)}1+\sum_{k\in \mathcal K_2(i)}\left(1-\frac{t_k}{\ell_i+t_i}\right)+\sum_{k\in \mathcal K_3(i)}\frac{\ell_k}{\ell_i+t_i}.
\]}

\begin{theorem}[He--Liao \cite{HeLiao2026}]
\label{thm:he-liao-diagonal-recurrence}
Suppose that each $\psi_i$ is positive and non-increasing and that
$\mathcal U(\boldsymbol\psi)$ is bounded. Then
\begin{equation}\label{eq:he-liao-recurrence}
 \dim_{\mathrm H}R(T_{\boldsymbol\beta},\boldsymbol\psi) = \sup_{\boldsymbol t\in\mathcal U(\boldsymbol\psi)} \min_{1\leq i\leq d}\theta_i(\boldsymbol t).
\end{equation}
\end{theorem}
\noindent If $\beta_i>1$ for all $i$, the monotonicity assumptions on the
functions $\psi_i$ may be omitted.

The theorem allows different {recurrent} rates in different coordinates and applies also when $\beta_i<-1$.  In particular, if $\psi_i(n)=\psi(n)$ for every $i$ and $\alpha=\alpha(\psi)<\infty$, then
\begin{equation}\label{eq:diagonal-recurrence}
 \dim_{\mathrm H}R(T_{\boldsymbol\beta},\boldsymbol\psi)=D_{\rm eig}(\alpha),
\end{equation}
where $D_{\rm eig}$ was defined in \eqref{eq:toral-eigenvalue-formula}. Thus, the {recurrent} sets and shrinking target sets for diagonal matrix
transformations have the same Hausdorff dimension under the
corresponding assumptions. Related results for hyperrectangular
targets centred at uniformly Lipschitz functions were obtained by Yuan
and Wang \cite{YuanWang2024}; {Recurrent} sets defined by multiplicative target conditions were also considered in \cite{HeLiao2026,YuanWang2024}.

We now turn to expanding integer matrices.  Let $A\in {GL_d}(\mathbb Z)$ be expanding, and order its eigenvalues so that 
\[1<|\lambda_1|\leq\cdots\leq|\lambda_d|.\]
Retain the notation
\[
 \ell_i=\log|\lambda_i|, \quad T_A \bfx=A \bfx~(\bmod~ 1).
\]
Let $\psi:\mathbb N\to(0,\infty)$ be decreasing, and let $\alpha=\alpha(\psi)$. 
Define
\[
 R(T_A,\psi) :=\{\bfx\in\mathbb T^d:T_A^n \bfx\in B(\bfx,\psi(n)) \text{ for infinitely many }n\}.
\]
For $n\ge1$, {rewrite
\begin{equation}\label{eq:recurrence-level}
 R_{n}(T_A,\psi): =\{\bfx\in\mathbb T^d:(A^n-I)\bfx~(\bmod{1})\in B({\bf 0},\psi(n))\}.
\end{equation} }
For all sufficiently large $n$, this is a union of $\lvert\det(A^n-I)\rvert\asymp \exp\left(n\sum_{i=1}^d\ell_i\right)$ translations of the ellipsoid $(A^n-I)^{-1}B({\bf 0},\psi(n))$. Their centres are the fixed points of $T_A^n$.

Let $A$ and $\psi$ be as above. Hu and Li \cite{HuLi2025} proved that, when
$\alpha\ge\ell_d-\ell_1$,
\begin{equation}\label{eq:expanding-recurrence}
 \dimH  R(T_A,\psi)  =D_{\rm cov}(\alpha),
\end{equation}
where $D_{\mathrm{cov}}$  was  { defined in \eqref{eq:toral-covering-bound}}

The matrix-sequence framework of Hu, Huang, Li and Wu
\cite{HuHuangLiWu2026}, reviewed in Section~3.4.2, applies to
recurrence by taking
\[
 M_n=A^n-I.
\]
There is no need to repeat the general construction here.  
In this formulation, the singular values of $(A^n-I)^{-1}$ determine the shapes of the recurrence
components, while  the successive minima of
\[
 (A^n-I)^{-1}\mathbb Z^d
\]
record the distribution of the fixed points of $T_A^n$.Their results
give $D_{\mathrm{eig}}(\alpha)$ as a general lower bound and a upper bound {associated with successive minima} that is no larger than $D_{\mathrm{cov}}(\alpha)$.
The same work yields exact recurrence formulae for expanding integer matrices diagonalizable over $\mathbb Q$ and for the rationally conjugate Jordan matrices considered in \eqref{eq:jordan-matrix-formula}.  In these cases, the Hausdorff dimensions of the {recurrent and} shrinking target  sets  agree. {To the best of our knowledge,} an exact formula for a general expanding integer matrix is not presently known.

\subsubsection{Hyperbolic integer matrices.}
Let $A\in {GL_2}(\mathbb Z)$, with eigenvalues
satisfying
\[
 0<|\lambda_1|<1<|\lambda_2|.
\]
Hu and Persson  \cite{HuPersson2024} first treated hyperbolic toral automorphisms and subsequently extended the result to nonsingular
$2\times2$ integer matrices with one contracting and one expanding
eigenvalue.

\begin{theorem}[Hu--Persson \cite{HuPersson2026}]
\label{thm:hu-persson-hyperbolic-recurrence}
Let $A$ satisfy the assumptions above. Let $\psi:\mathbb N\to(0,\infty)$ be decreasing. Then, for every $\alpha>0$,
\begin{equation}\label{eq:hyperbolic-recurrence}
 \dim_{\mathrm H}R(T_A,\psi) =
 \min\left\{ \frac{2\ell_2}{\alpha+\ell_2},
 \frac{\ell_2}{\alpha} \right\}.
\end{equation}
\end{theorem}

Unlike the corresponding dimension formula for shrinking target sets in Section~3.4.3, \eqref{eq:hyperbolic-recurrence} is positive for any
$\alpha>0$ and depends only on the expanding eigenvalue $\lambda_2$.

\section{Shrinking targets versus recurrence}\label{sec:comparison}

This section compares the Hausdorff dimensions of {the} shrinking target sets and {recurrent} sets when the two problems are defined using the same
approximation function. We restrict attention to dimension results; for
quantitative counting laws, see \cite{LevesleyLiSimmonsVelani2025}.

\subsection{Two unifying frameworks}

The preceding results show that {the recurrent and shrinking target sets} often have the same Hausdorff dimension in expanding systems, but that this agreement is not universal. Two formulations help clarify the relation between the two problems.

Let $T:X\to X$ be a dynamical system on a metric space $(X,\bfd)$, and let
$\boldsymbol g=(g_n)_{n\geq1}$ be a sequence of Lipschitz maps from $X$ to
itself. Define
\begin{equation}\label{eq:twisted-recurrence}
 R^{\boldsymbol g}(\psi):= \left\{x\in X: \bfd(T^n x,g_n(x))<\psi(n) \ \text{for infinitely many }n \right\}.
\end{equation}
This formulation contains both shrinking target sets and  { recurrent} sets.  If $g_n(x)\equiv y$, then \eqref{eq:twisted-recurrence} is the shrinking target set $W(T,y,\psi)$, whereas $g_n=\operatorname{id}$ gives the usual { recurrent} set. %Kleinbock and Zheng \cite{KleinbockZheng2023} established dynamical Borel--Cantelli lemma for $ R_T^g(\psi)$ under suitable expansion, regularity and mixing assumptions. 
For a fixed piecewise Lipschitz map $g_n=g$, L{\"u}, Wang and Wu~\cite{LuWangWu2025} established a Lebesgue-measure zero--one law for $\beta$-transformations. Kleinbock and Zheng~\cite{KleinbockZheng2023} established a more general dynamical Borel--Cantelli lemma for $ R^g(\psi)$ under suitable assumptions.

Dimension results in this framework were obtained in several expanding dynamical systems. For the conformal setting, Shen~\cite{ShenModifiedSelfConformal2023}
obtained an explicit formula for the twisted recurrence set. Let
$\Phi$ be a finite conformal iterated function system satisfying the
open set condition, let $J$ be its attractor. For a Lipschitz map $g:J\to J$
and a positive function $\psi$,
\[
 \dimH R^g(\psi)= \begin{cases}
   s_\psi, & 0\leq\alpha(\psi)<\infty,\\
   0,      & \alpha(\psi)=\infty,
 \end{cases}
\]
where $s_\psi$ is the unique solution of
\[
 P_\Phi(s_\psi)-\alpha(\psi)s_\psi=0.
\]
In particular, taking $g\equiv y$ and $g=\operatorname{id}$, respectively, gives
\[
 \dimH W(T,y,\psi)=\dimH R(T,\psi).
\]
Thus, in this setting, the dimension is independent of the particular
Lipschitz target map.
Moreover, Wang \cite{WangModifiedBeta2018} treated $\beta$-transformations, and  Yuan and Li \cite{YuanLiLipschitz2023} studied expanding Markov maps. 
Higher-dimensional analogues for expanding matrix transformations and rectangular targets were obtained by Yuan and Wang \cite{YuanWang2024}. These results have been described and are not repeated here. 
The Hausdorff-measure theory was further developed by He \cite{HeHausdorffMeasure2025} for products of $\beta$-transformations, and  obtained zero--full laws for the Hausdorff $f$-measure.

For matrix transformations, Hu, Huang, Li and Wu \cite{HuHuangLiWu2026} obtained a different unification based on a sequence of invertible matrices $\cM=(M_n)_{n\geq1}\subset GL_d(\mathbb R)$.  Their general limsup set is
\[
 W(\cM,\psi,\bfy)=\left\{ \bfx\in[0,1)^d: M_n\bfx\pmod1\in B(\bfy,\psi(n)) \ \text{for infinitely many }n \right\}.
\]
Let $A$ be a non-singular integer matrix. For $T_Ax=Ax\pmod1$, the choices $M_n=A^n$ and $M_n=A^n-I$, with $y=0$ in the latter case, give shrinking target sets and { recurrent} sets, respectively.  Their dimension estimates for
such matrix-sequence limsup sets use the singular values of $M_n$ together with the successive minima of $M_n^{-1}\mathbb Z^d$.  The matrix-sequence formulation therefore places the two problems in a common linear framework; the detailed bounds were given in the preceding sections.

For an expanding integer matrix $A$, the sequences $A^n$ and
$A^n-I$ have the same exponential singular-value and determinant
growth. This explains why shrinking targets and recurrence have the
same Hausdorff dimension in several structured classes of expanding
matrices. Such an equality is not automatic, however, since the
distributions of the lattices $A^{-n}\mathbb Z^d$ and
$(A^n-I)^{-1}\mathbb Z^d$ may be different.

These two frameworks unify the studies of shrinking target sets and
 {recurrent} sets, but they do not by themselves imply that the corresponding Hausdorff dimensions are always equal.

\subsection{Comparison of Hausdorff dimensions}

Let $W$ and $R$ denote the corresponding shrinking target sets and
{recurrent} sets, respectively. 
For a positive function $\psi:\mathbb N\to(0,\infty)$, let $\alpha=\alpha(\psi)$ be its lower exponential order at infinity. Table~\ref{tab:comparison} records representative comparisons in which the shrinking rates and target formulations are matched. The precise hypotheses are given in Sections~\ref{sec:shrinking} and
\ref{sec:recurrence}.

\begin{table}[ht]
\centering
\begingroup
\small
\setlength{\tabcolsep}{4pt}
\renewcommand{\arraystretch}{1.15}

\begin{tabularx}{\linewidth}{
  @{}
  >{\raggedright\arraybackslash}p{0.20\linewidth}
  >{\raggedright\arraybackslash}p{0.45\linewidth}
  >{\raggedright\arraybackslash}X
  @{}
}
\toprule
System & Dimension results & Comparison\\
\midrule

$\beta$-transformations & For $\beta>1$,
$\displaystyle
\dimH W=\dimH R
=\frac{\log\beta}{\alpha+\log\beta}$, both dimensions are zero if \(\alpha=\infty\).
& 
\emph{Agreement} for the same approximation function.
\\

\addlinespace[2pt]
Expanding Markov maps & For finitely many branches and \(0\leq\alpha<\infty\), \(\dim_{\mathrm H}W=\dim_{\mathrm H}R=s_0,\)
where \(s_0\) is the unique zero of $P\!\left(- s\bigl(\log|T'|+\alpha\bigr)\right)$, and the corresponding pressure infimum
for countably many branches.
&
\emph{Agreement} under the stated hypotheses.
\\

\addlinespace[2pt]
Finite conformal IFSs
&
For $0\leq\alpha<\infty$,
$\dimH W=\dimH R=s_\psi$, $P_\Phi(s_\psi)-\alpha(\psi)s_\psi=0,$ both dimensions are zero if \(\alpha=\infty\).
&
\emph{Agreement} for the common time-dependent radius
$\psi(n)$, under the open set condition.
\\

\addlinespace[2pt]
Expanding diagonal transformations
&
%For matching rectangular rates \(\Psi=(\psi_1,\ldots,\psi_d)\),
Under the common hypotheses,
$\dimH W=\dimH R=D_{\rm eig}.$
% \sup_{\boldsymbol t\in\mathcal U(\Psi)} \min_{1\leq i\leq d}\vartheta_i(\boldsymbol t)$
 %under the common hypotheses.
&
\emph{Agreement} for the same  approximation functions.
\\

\addlinespace[2pt]
Expanding integer matrices
&
%\(\displaystyle D_{\rm eig}\leq\dimH W\leq D_{\rm lat}\), \par
%\(\displaystyle D_{\rm eig}\leq\dimH R\leq D_{\rm rec}\).
For \(\alpha>\ell_d-\ell_1\), 
$\dimH W=\dimH R=D_{\rm cov}.$
The dimensions also agree for matrices diagonalizable over
\(\mathbb Q\) and in the rationally conjugate Jordan cases.
&
\emph{Agreement} in these structured regimes; open in general.
\\

\addlinespace[2pt]
Irrational rotations
&
For \(\psi(n)=n^{-\gamma}\), \(\gamma\geq1\),
\(\dimH W=1/\gamma\), whereas \(\dimH R\in\{0,1\}\).
&
\emph{Possible disagreement} for \(\gamma>1\)
\\

\addlinespace[2pt]
Hyperbolic toral automorphisms
&
For \(0<\alpha<\ell_2\),
\(\dimH W=\dimH R=2\ell_2/(\alpha+\ell_2)\).
\par
For \(\alpha>\ell_2\),
\(\dimH W=0<\dimH R=\ell_2/\alpha\).
&
\emph{Agreement} below \(\ell_2\), but \emph{disagreement}
above it. The critical case is centre-dependent.
\\

\bottomrule
\end{tabularx}
\endgroup

\caption{Representative comparisons between shrinking target sets and
{recurrent} sets.}
\label{tab:comparison}
\end{table}

The affine results surveyed above do not presently yield a general
comparison based on matched target formulations. Koivusalo--Ram\'irez \cite{KoivusaloRamirez2018}
and Koivusalo--Liao--Morris--Rams\cite{KoivusaloLiaoRams2022} treat symbolic shrinking targets, whereas B\'ar\'any and Troscheit's shrinking target theorem  concerns prescribed cylinders and their recurrence theorem concerns returns to
the initial blocks of the orbit \cite{BaranyTroscheit2021}. These two results involve different pressure quantities and imply neither a general equality nor a general ordering. For fixed planar systems, B\'ar\'any- Rams\cite{BaranyRams2018} and Jordan--Koivusalo \cite{JordanKoivusalo2025} study shrinking target sets, while Wu--Yuan \cite{WuYuan2025} study geometric recurrence on a restricted class of Bedford--McMullen carpets. Baker and Koivusalo \cite{BakerKoivusalo2024} place shrinking target sets and {recurrent} sets in a common inverse-branch framework, but their general results are mainly measure-theoretic and do not provide matching exact dimension formulae.

For expanding integer matrices, the comparison is instead governed by
the relation between $A^n$ and $A^n-I$. Their exponential
singular-value growth agrees, which accounts for the equality obtained
in several structured cases, but the relevant lattices may have
different distributions. The general expanding case therefore remains
open. Irrational rotations and hyperbolic toral automorphisms show that, outside the uniformly expanding setting, the two dimensions may differ.

\section{Badly approximating orbits, non-recurrent sets and intersection with fractals}\label{badsec}
Let $(X,\mathbf{d})$ be a metric space and let $T:X\to X$. Given $y\in X$, we denote $\mathrm{BAD}(T,y)$ as the set of $x\in X$ whose orbits under $T$ badly approximate $y$, that is
\begin{equation*}
    \mathrm{BAD}(T,y):=\left\{x\in X: \liminf_{n\to\infty}\mathbf{d}(T^n(x),y)>0\right\}.
\end{equation*}
The associated non-recurrent set of $T$ is defined as
\begin{equation*}
    \mathrm{N}(T):=\left\{x\in X:\liminf_{n\to\infty}\mathbf{d}(T^n(x),x)>0\right\}.
\end{equation*}
When the underlying space is taken to be $X=[0,1)^d$, we shall write $\Bad_d(T,\bfy)$ and $\tN_d(T)$ instead. 
In this section, we aim at providing a concise overview of known developments on dimension theory of $\Bad(T,y)$ and $\tN(T)$, with a main focus on the two canonical cases where $T$ is either a piecewise $C^{1+\kappa}$ expanding interval map or a matrix transformation.  For outstanding results of other interesting and important dynamics not covered in the paper, such as expanding maps on manifolds or bounded flows for homogeneous dynamics, see \cite{MR3987817,MR3436158,MR3951832,MR4444693,MR1107012} and references therein.

\subsection{Matrix sequences}\label{matrseq} In this part, we consider the case $T=T_A$ where $A$ is a $d\times d$ integral matrix. As in the previous section, the sets of interest $\Bad_d(T_A,\bfy)$ and $\tN_d(T_A)$ are unified via
\begin{equation*}
    \Bad_{m,d}(\cM,\cY):=\left\{\bfx\in[0,1)^d:\liminf_{n\to\infty}\dist(T_{M_n}(\bfx)-\bfy_n,\Z^m)>0\right\},
\end{equation*}
where $\cM:=\{M_n\}_{n\in\mathbb{N}}$ with each $M_n$ a $m\times d$ real matrix and $\cY:=\{\bfy_n\}_{n\in\mathbb{N}}$ is a sequence of points in $[0,1)^m$. We would use the notation $\Bad_d(\cM,\cY)$ for the case where $m=d$.  When $A$ is integral, the set $\Bad_d(\cM,\cY)$ coincides with $\Bad_d(T_A,\bfy)$ by taking $M_n=A^n$ and $\bfy_n=\bfy$, and it equals $\tN_d(T_A)$ when $M_n=A^n-I$ and $\bfy_n=0$.

The  set $\Bad_{m,d}(\cM,\bfy)$ is related to the non-dense orbits of $\cM$ defined as 
\begin{equation*}
\ND_{m,d}(\cM)\ := \ \left\{\bfx\in[0,1)^d:\{M_n\bfx\ \mathrm{mod} \ 1\}_{n\in\N}\ \text{is not dense in}\ [0,1)^m \right\}.
\end{equation*}
Indeed, $\ND_{m,d}(\cM)$ can be expressed as a countable union of sets $\Bad_{m,d}(\cM,\bfy)$ over $\bfy\in\Q^m\cap[0,1)^m$. We would write $\ND_d(\cM)$ for the case where $m=d$.
The study of non-dense orbits associated with lacunary sequences has  a long history. Let us first discuss the one-dimensional setting. We say that a sequence $\cM=\{M_n\}_{n\in\N}$ of real numbers is lacunary if $\inf_{n\in\N} M_{n+1}/M_n>1$. In 1975, Erd\"{o}s \cite{MR412138} asked whether $\ND_1(\cM)\cap\Q^c\neq\varnothing$ for any lacunary sequence $\cM=\{M_n\}_{n\in\N}$ of integers. The problem was first completely resolved by Pollington \cite{MR540398},  who showed that $\Bad_1(\cM,0)$ has full Hausdorff dimension when $\cM$ is a lacunary sequence of real numbers (not necessarily integers). Regarding general sequence $\cY$, Broderick, Begeuad, Fishmen, Kleinbock and Weiss \cite{MR2644378} established the full Hausdorff dimension of $\Bad_1(\cM,\cY)$ when $\cM$ is a lacunary sequence and $\cY$ is a sequence in $[0,1)$.

Regarding the high dimensional case, Dani \cite{MR980795} has shown in 1988 that $\dim_H(\Bad_d(T_A,\bfy))=d$ whenever $A\in GL_d(\Z)$ is semisimple and $\bfy\in\Q^d\cap[0,1)^d$. The result was then generalized by Broderick, Fishman and Kleinbock \cite{MR2818688} to more general $\cM$ and $\cY$. To state their theorem accurately, we say that a sequence $\cM=\{M_n\}_{n\in\N}$ of $d\times d$ matrices is lacunary if their operator norms $\{\|M_n\|\}_{n\in\N}$ is lacunary, where 
\begin{equation*}
    \|A\|:=\max_{|\bfx|=1}|A\bfx|.
\end{equation*}
They showed that $\dimH(\Bad_{m,d}(\cM,\cY))=d$ whenever $\cM$ is a lacunary sequence of $m\times d$ matrices and $\cY$ is a sequence in $[0,1)^m$. By applying this result to $M_n=A^n$ or $M_n=A^n-I$, we can derive that $\Bad_d(T_A,\bfy)$ and $\tN_d(T_A)$ has full Hausdorff dimension if $\bfy\in[0,1)^d$ and $A$ is a $d\times d$ integral matrix with at least one eigenvalue whose absolutely value is greater than one.

 Although not directly related to the topic of dynamics in the present paper, it is natural to ask what the Hausdorff dimensio of $\Bad_{m,d}(\cM,\cY)$ is when $\cM$ is not lacunary. For one-dimensional setting, there is a complete result for sublacunary sequences. We say that a sequence $\cM$ of real numbers is sublacunary if $M_{n+1}/M_n\to 1$ and $M_n\to +\infty$ as $n\to\infty$. In 1994, Boshernitzan \cite{MR1293584} proved that $\dimH(\ND_1(\cM))=0$ whenever $\cM$ is a sublacunary sequence of positive integers\footnote{In the course of calculating $\dimH(\ND_1(\cM))$, Boshernitzan employed the notion of limit sets, which can be viewed as an early precursor to micro sets and weak tangents. These later two notions have seen substantial development and wide applicability in recent years (see \cite{MR3819712,MR4011524,MR2408385,MR3235463,MR4428883,MR2662522,MR4313141} and references therein).}. He also asserted, without proof, that a slight modification of his argument would give $\dimH(\Bad_1(\cM,\cY))=0$ for the same $\cM$ and any sequence $\cY$ in the unit interval. However, the formula of Hausdorff dimension of $\Bad_1(\cM,\cY)$ remains not known  when $\cM$ is neither lacunary nor sublacunary. At the most coarse level, we only know that $\dimH(\ND_1(\cM))$ has intermediate property \cite{berend2024}; that is to say that for any $s\in[0,1]$, there exists a sequence $\cM$ of positive real numbers with $M_n\to\infty$ ($n\to\infty$) such that $\dimH(\ND_1(\cM))=s$.

  In contrast to one dimensional setting, for a sequence $\cM$ of $d\times d$ real matrices, the Hausdorff dimension of $\Bad_d(\cM,\bfy)$ may depend on $\bfy$ even when $M_n=a_n I$ for a sublacunary sequence $\{a_n\}_{n\in\N}$ . The following simple example illustrates this point. In general, however, the precise mechanism by which  $\bfy\in[0,1)^d$ affects $\dimH(\Bad_d(\cM,\bfy))$ remains unknown.
 \begin{example}\label{exmaple}
     Let $d\geq2$ and let $M_n=nI$ for any $n\in\N$. Then $\Bad_d(\cM,0)=\varnothing$ and $\dimH(\Bad_d(\cM,\bfy))\geq d-1$ when $\bfy\neq 0$.
 \end{example}
 \begin{proof}
     By Dirichlet's theorem,  one has $$\liminf_{n\to\infty}\dist(n\bm{\alpha},\Z^d)=0,\qquad\forall \ \bm{\alpha}\in[0,1)^d.$$ This proves the first statement. If $\bfy=(y_1,y_2,\cdots,y_d)\in[0,1)^d\setminus\{0\}$, say $y_1\neq 0$, then $\Bad_d(\cM,\bfy)\supseteq\{0\}\times[0,1)^{d-1}$. This implies the second statement.
 \end{proof}
 We now consider the Hausdorff dimension of $\ND_d(\cM)$.  A high dimensional analogue of Boshernitzan's theorem is due to Baker \cite{MR2803785}, who showed that $\dimH (\ND_d(\cM))=d-1$ if $M_n=a_n I$ and $\{a_n\}_{n\in\N}$ is a sublacunary sequence of positive real numbers. 
  % In particular, this implies that the sets in Example \ref{exmaple} have Hausdorff dimension $d-1$ when $\bfy\neq0$. 
  The proof combines Boshernitzan's argument and the slicing theorem from geometric measure theory. However, this method relies heavily on the assumption that $M_n$ is a scalar matrix and does not appear to be readily adaptable to general diagonal matrices whose entries are sublacunary sequences of positive integers. To date, no result concerning Hausdorff dimension of $\ND_d(\cM)$ has been established when $\cM$ is neither lacunary nor diagonal type with sublacunary entries, not even an intermediate property analogous to that known in one-dimensional case. We conclude this subsection by stating the following problems that remain unsolved, at least not  formally, in the present literature.
 \begin{problem}
     For any integer $1\leq i\leq d$, let $\{a_{i,n}\}_{n\in\mathbb{N}}$ be a real sequence such that $a_{i,n}\to +\infty$ and $a_{i,n}/a_{i,n+1}\to 1$ as $n\to\infty$. Let $M_n:=\mathrm{diag}(a_{1,n},a_{2,n},\cdots, a_{d,n})$ for any $n\in\mathbb{N}$. Is it true that $\dim_H(\ND_d(\mathcal{M}))=d-1$\,?
 \end{problem}
 \begin{problem}
     For any $s\in[0,d]$, is there a sequence $\cM$ of $d\times d$ real matrices with $\|M_n\|\to\infty$ that satisfies $\dimH(\ND_d(\cM))=s$\,?
 \end{problem}

 \subsection{Piecewise expanding maps}\label{PEM} We discuss dimension theory of $\Bad_1(T,y)$ and $\tN_1(T)$ for piecewise $C^{1+\kappa}$ expanding maps on $[0,1)$ in this subsection. In contrast to integral matrix transformations mainly considered previously, the iterates of such maps are generally much more complicated and may not admit a convenient structure such as the Markov partition. We first introduce the standard definition of piecewise $C^{1+\kappa}$ expanding map.
 \begin{definition}\label{defofpiee}
   Given $\kappa>0$ and a subset $K\subseteq[0,1)$, we say that a map $T:[0,1)\to[0,1)$ is \emph{piecewise $C^{1+\kappa}$  expanding} on $K$ if it satisfies the following statements:
    \begin{itemize}
        \item[(i)] $T$ is \emph{piecewise $C^1$ monotonic}; that is to say that there exist a finite or infinite alphabet $\Sigma\subset\N$ and a collection of disjoint open sub-intervals $\{I_i\}_{i\in\Sigma}$ in $[0,1)$ satisfying that $$\overline{\bigcup_{i\in\Sigma}{I}_i}\ \supseteq  K$$ and that $T|_{K\cap I_i}$ cap be extended to a  monotonic and  $C^1$ map on  $I_i$ for any $i\in\Sigma$. 
        \vspace{1ex}

        \item[(ii)] there exist $\eta>1$ and $N_0\in\N$ such that $ |(T^{N_0})'(x)|>\eta$ for all $x\in\bigcup_{i\in\Sigma} I_i$ at which $T^{N_0}$ is well-defined and differentiable.\vspace{1ex}
     
        \item[(iii)] there exists $C>0$  such that
        \begin{eqnarray*}
        |T'(x)-T'(y)|\leq C|x-y|^{\kappa}
        \end{eqnarray*}
for any $i\in\Sigma$ and any $x,y\in I_i$.\vspace{1ex}
    \end{itemize}
 We say that \emph{$T$ has finite branches} if $\#\Sigma<+\infty$ whereas \emph{infinite branches} if $\#\Sigma=+\infty$. In addition, $T$ is a Markov map if for any $i,j\in\Sigma$ with $I_i\cap f(I_j)\neq\varnothing$ we have $I_i\subseteq f(I_j)$, and has full branches if $T(I_i)\supseteq(0,1)$ for any $i\in\Sigma$.
\end{definition}
\begin{remark}\label{listmap}
    Some typical interval maps that satisfy Definition \ref{defofpiee} are in order:
    \begin{itemize}
        \item[(i)] $\beta$ and negative $\beta$-transformation: $T_{\beta}(x):=\beta x \  (\mathrm{mod}\ 1)$, where $|\beta|>1$;\vspace{1ex}
        \item[(ii)] $(\alpha,\beta)$-transformation: $T_{\beta,\alpha}(x):=\beta x+\alpha\ (\mathrm{mod}\ 1)$, where $|\beta|>1$ and $\alpha\in[0,1)$;\vspace{1ex}
        \item[(iii)] Gauss map: $T_G(x):=\frac{1}{x}\ (\mathrm{mod}\ 1)$;\vspace{1ex}
        \item[(iv)] L\"{u}roth transformation: $T_L(x):=(i+1)(i+2)x-i-1$, $\forall \ x\in[1/(i+2),1/(i+1))$.
    \end{itemize}
\end{remark}

 In 1997, Dolgopyat \cite{MR1440064} studied sets of $x\in[0,1)$ whose orbits avoid prescribed subsets of $[0,1)$. To state the result specifically, given $Y\subseteq X$ and $x\in X$, let $\orb(T,x):=\{T^n(x):n\in\N\}$ and let
\begin{equation}\label{defofety}
    E(T,Y)\ := \ \left\{x\in X:\text{limit points of $\orb(T,x)$ do not intersect $Y$}\right\}.
\end{equation}
Dolgopyat proved that if $T$ is a piecewise $C^{1+\kappa}$ expanding map with finite branches on $[0,1)$ and $Y\subseteq[0,1)$ has Hausdorff dimension less than one, then $\dimH(E(T,Y))=1$. In particular, since $E(T,\{y\})= \Bad_1(T,y)$ for any $y\in[0,1)$, this result implies  that $\dimH(\Bad_1(T,y))=1$. His approach is highly heuristic in nature. At the most coarse level, it can be broken down into the following three steps:
\begin{itemize}
    \item[(i)] for any $\varepsilon>0$, find a bounded function $n:[0,1)\to\N$ and a $g$-invariant subset $K\subseteq[0,1)$ with Hausdorff dimension larger than $1-\varepsilon$ such that $g|_K:K\to K$ is a Markov map, where $g(x):= T^{n(x)}(x)$;\vspace{1ex}
    \item[(ii)] establish $\dimH(E(g|_K,Y))=\dimH(K)$ when $\dimH (Y)<\dimH(K)$;\vspace{1ex}
    \item[(iii)] prove that $\dimH(E(T,Y))=1$ whenever $\dimH(Y)<1$ by combining (i)-(ii) and a direct observation $E(T,Y)=E(g,Y_N)$, where $N=\max\{n(x):x\in[0,1)\}$ and $Y_N$ is the union of $T^{-i}(Y)$ over $i=0,1,...,N-1$.
\end{itemize}
Dolgopyat then employed a modified version of this strategy to obtain a corresponding result for Anosov flows. This approach, particularly the approximation of a general map by a Markov map as in (i), proves to be highly powerful. For example, the authors in \cite{LiLiaoVelaniZorin2023} adopted the technique to study the Hausdorff dimensions of Shrinking target sets associated with piecewise linear maps. We refer to the last two paragraphs of Section 1.2 in \cite{MR3951832} for more detailed discussions on Dolgopyat's argument.

At first glance, non-recurrent sets differ greatly from the sets discussed above, thus one might think  Dolgopyat's theorem cannot be applied to $\tN_1(T)$. However, Yuan, Li and Wu \cite{nayuan2021} bridge this gap in Section 4 of their paper,  by observing that $\tN_1(T)$ locally resembles $\Bad_1(T,y)$ and that enlarging local regions does not affect the Hausdorff dimension. Consequently, though not stated formally, their  arguments, together with Dolgopyat's result, indeed implies the full Hausdorff dimension of $\tN_1(T)$ for piecewise $C^{1+\kappa}$ expanding interval maps with finite branches. Moreover, they independently obtain the full dimension for some class of Markov maps with infinite branches, including Gauss map and L\"{u}roth map listed in Remark \ref{listmap}. In addition, they  established, via symbolic dynamics,  the full dimension of the set
\begin{equation*}
    \Bad_1(T,\cY)\ := \ \left\{x\in[0,1):\liminf_{n\to\infty}|T^nx-y_n|>0\right\}\qquad (\cY=\{y_n\}_{n\in\N}\subseteq[0,1))
\end{equation*}
which generalized results of \cite{liu2019,MR3794083,MR3957180}. One might initially identify $\Bad_1(T,\cY)$ with $E(T,Y)$ defined in \eqref{defofety} while taking $Y=\{y_n:n\in\N\}$. They differ, however, in two essential respects: (i) $\Bad_1(T,\cY)$ requires each $y_n$ to be avoided only by the corresponding iterate $T^n(x)$, whereas $E(T,Y)$ requires every $y\in Y$ to be avoided eventually by the entire orbit; (ii) the avoiding distance in $\Bad_1(T,\cY)$ is uniformly bounded away from zero, while for $E(T,Y)$ the distance may depend on the particular $y\in Y$. The upshot is that the full dimension of $\Bad_1(T,\cY)$ does not directly follow from that of $E(T,Y)$. Later, Yuan and Wang \cite{MR4848062} extended the setting of previous work {\cite{nayuan2021}} to repellers associated with some certain kind of Markov expanding maps, thereby allowing dynamics defined on some fractals such as cookie-cutters.

A common way to unify $\Bad(T,y)$, $\Bad(T,\cY)$ and $\tN(T)$ for general dynamics is to introduce a sequence of functions $\{g_n\}_{n\in\N}$. More precisely, for $\cG=\{g_{n}:X\to X\}_{n\in\N}$, define
\begin{equation*}
    \tN(T,\cG)\ := \ \left\{x\in X:\liminf_{n\to\infty}\mathbf{d}(T^n(x),g_n(x))>0\right\}
\end{equation*}
and we call $\tN(T,\cG)$ a twisted non-recurrent set with respect to $\cG$. As usual, we would write $\tN_d(T,\cG)$ instead when $X=[0,1)^d$. Thanks to recent progress, the dimension theory of $\tN(T,\cG)$ is rather complete for one-dimensional piecewise expanding maps. In particular, Lambert, Simmons and Zheng  \cite{lambert2025} have proved $\dimH(\tN_1(T,\cG))=1$ for $T=T_{\beta}$ ($\beta>1$) and $T=T_{G}$ with $\cG$ consisting of equicontinuous functions on $[0,1)$. Simultaneously and independently, Huang, Li, Wang and Yuan \cite{huang2025win} obtained the same result for $T$ being a piecewise $C^{1+\kappa}$ expanding map with finite branches or infinitely full branches and $\cG$ being a sequence of uniformly Lipshitz maps, that is
\begin{equation*}
    |g_n(x)-g_n(y)|\ \leq \ C|x-y|,\qquad\forall \ n\in\N,\quad\forall \ x,y\in[0,1)
\end{equation*}
for some constant $C>0$.

 \subsection{Winning theory and intersection with fractals} As seen above, for dynamics considered in the previous two subsections, the sets $\Bad(T,y)$ and $\tN(T)$ are of full Hausdorff dimension. In this part, we will see a stronger phenomenon:  they are much larger than merely being full dimensional and they intersect a class of fractals in sets of positive dimension. To begin with, we recall Schmidt's game as introduced in \cite{MR195595}. For other variants of games not mentioned in this section, such as rapid game, see \cite{hatefi202,hussainsimmon}. Let $\alpha,\gamma\in(0,1)$ and let $K\subseteq\R^d$ be a closed subset. The $(\alpha,\gamma)$-game on $K$ is defined as follows:  Two players, Alice and Bob, alternately pick closed balls whose centers lie in $K$, so that the chosen balls form a nested decreasing sequence
\begin{equation*}
    B_1\supseteq A_1\supseteq B_2\supseteq A_2\supseteq\cdots\supseteq B_k\supseteq A_k\supseteq\cdots
\end{equation*}
with  diameters
\begin{equation*}
    |A_k|=\alpha\cdot|B_k|,\quad|B_{k+1}|=\gamma\cdot|A_k|,\quad\forall \ k\in\N,
\end{equation*}
where $A_k$ and $B_k$ are the $k$-th moves of Alice and Bob respectively, and $|E|$ denotes the diameter of the set $E$.  A set $S\subseteq \R^d$ is said to be $(\alpha,\gamma)$-winning on $K$ if Alice possesses a strategy of choosing her balls that guarantees $$\bigcap_{k=1}^{\infty}B_k\subseteq S$$ no matter how Bob selects his balls. It is called $\alpha$-winning on $K$ if it is $(\alpha,\gamma)$-winning on $K$ for any $\gamma\in(0,1)$. We simply say that $S$ is winning on $K$ if it is $\alpha$-winning on $K$ for some $\alpha\in(0,1)$. Below we collect several notable properties of these $\alpha$-winning sets.
\begin{proposition}[{\cite{MR195595}}]\label{winnpro}
    Let $E$, $E_n$ be subsets of $\R^d$ for any $n\in\N$, let $K\subseteq\R^d$ be a closed subset with non-empty interior and let $\alpha\in(0,1)$.
    \begin{itemize}
        \item[(i)] If $E$ is $\alpha$-winning on $K$, then $\dimH E=d$.
 \medskip
 \item[(ii)] If $E$ is $\alpha$-winning on $K$, then it is $\alpha'$-winning on $K$ for any $0<\alpha'\leq\alpha$.
 \medskip

 \item[(iii)] If $E_n$ is $\alpha$-winning on $K$ for any $n\in\N$, then $\bigcap_{n=1}^{\infty}E_n$ is $\alpha$-winning on $K$.\medskip

 \item[(iv)] If $E\neq\R^d$ and $\alpha>1/2$, then $E$ is not $\alpha$-winning on $\R^d$.
    \end{itemize}
\end{proposition}

In many common dynamical systems, the sets $\Bad(T,y)$ and $\tN(T)$ are $\alpha$-winning. The first such result is due to Schmidt  \cite{MR195595}, who showed that $\Bad_1(T_G,0)$ is  $1/2$-winning. After this influential work, a rich literature concerning the winning property of $\BAD(T,y)$ has emerged. For one-dimensional dynamics, Tseng \cite{MR2480100} established that for any $C^2$ expanding self-map $T$ on the circle $S^1$, the set $\BAD(T,y)$ is $1/2$-winning. Beyond the smooth case, a collection of works \cite{MR2739401,MR2660561,MR3206688,MR3021798} have verified that for $T$ being a $\beta$-transformation, the Gauss map, or the L\"{u}roth transformation on the unit interval, the $1/2$-winning property of $\Bad_1(T,y)$ holds  for every $y\in[0,1)$. Nevertheless, one may observe that most of these approaches often employ fairly elaborate symbolic combinatorics, and their applicability is often confined to certain specific settings.

A recent advance is due to Hu, Li and Yu \cite{hu2025}, who have considerably extended the $\beta$-transformation setting by showing that for any piecewise $C^{1+\kappa}$ expanding map $T$ on $[0,1)$ with finitely many branches (see Definition \ref{defofpiee}), the set $\BAD_1(T,y)$ is $1/2$-winning for every $y\in[0,1)$. The result is substantially stronger, yet the underlying method is surprisingly simple. One key observation they rely on is that $T^n$ is monotone and $C^{1+\kappa}$ on an interval if it contains no endpoints of $n$-cylinders. For non-expanding systems, Duvall \cite{MR4157863} proved that if $T$ belongs to a certain class of nonuniformly expanding maps on $[0,1)$, then $\BAD_1(T,0)$ is $\alpha$-winning for some $\alpha\in(0,1)$. 

Concerning high dimensional systems, Dani \cite{MR980795} established that for any semisimple $A\in GL(d,\mathbb Z)$ and any $\mathbf y\in\mathbb Q^d$, the set $\Bad_d(T_A,\mathbf y)$ is $1/2$-winning. It was generalized by Broderick, Fishman and Kleinbock \cite{MR2818688} who covered all $\mathbf y\in[0,1)^d$ and all $A\in GL(d,\mathbb Z)$ for which at least one eigenvalue has modulus exceeding one. In fact, their work demonstrates that $\Bad_{m,d}(\mathcal M,\mathcal Y)$ is winning for every lacunary sequence $\mathcal M$ and every sequence of points $\mathcal Y$ in $[0,1)^m$. As a consequence, their framework also includes the non-recurrent sets $\tN_d(T_A)$. For analogous results on partially hyperbolic diffeomorphisms or homogeneous spaces, we refer the reader to \cite{MR4389795,MR3519428} and the references cited therein.

We are particularly interested in the Hausdorff dimension of $\tN_d(T,\cG)\cap K$ when $K\subseteq[0,1]^d$ is a fractal set. This motivates the winning theory on fractal sets.  As shown in Proposition~\ref{winnpro} (i), whenever the ambient space has nonempty interior, winning sets are dimensionally large. This naturally leads one to conjecture that if $K\subseteq\mathbb R^d$ is closed with positive Hausdorff dimension and $E\subseteq K$ is winning on $K$, then $E$ itself should inherit a substantial dimension. The following proposition offers a partial affirmative answer to this conjecture. To state it with precision, we introduce the notion of lower dimension. For a set $E\subseteq\mathbb R^d$, its lower dimension is defined by
\[
\dimL E := \sup \left\{ s\ge 0 \;\middle|\;
\begin{aligned}
&\exists\, c>0 \text{ such that for all } \mathbf x\in E\ \text{and}\ \varepsilon,\rho\in(0,1),\\
&\text{it holds that}\ N_{\varepsilon\rho}\bigl(B(\mathbf x,\rho)\cap E\bigr) \ge c\,\varepsilon^{-s}
\end{aligned}
\right\},
\]
where $N_r(E)$ denotes the least number of sets of diameter at most $r$ required to cover $E$.

\begin{proposition}[{\cite[Proposition 2.5]{MR3904182}}]\label{dimhlowdima}
For every closed set $K\subseteq\mathbb R^d$, any winning set $E\subseteq K$ on $K$ satisfies $\dimH E \ge \dimL K.$
\end{proposition}

A suitable measure supported on $K$ facilitates playing Schmidt's game. A Borel probability measure $\mu$ on $\mathbb R^d$ is called absolutely decaying if one can find constants $C>0$, $r_0>0$ and $s>0$ with the following property: for every $\mathbf x\in\operatorname{supp}(\mu)$,  $0<r\le r_0$,  $0<\gamma\le 1$, and affine hyperplane $\mathcal L$ of codimension one, the following  inequality holds
\[
\mu\bigl(B(\mathbf x,r)\cap(\mathcal L)_{\gamma r}\bigr)
\le C\,\gamma^s\,\mu(B(\mathbf x,r)).
\]
In the particular case $d=1$, $\mu$ is absolutely decaying if the same estimate is satisfied with $(\cL)_{\gamma r}$ replaced by $B(y,\gamma r)$ for every $y\in\R$. Under this decay condition, it has been shown in \cite{MR2644378,MR2818688} that $\Bad_d(\cM,\cY)$ is winning on $K$ when $K=\supp(\mu)$ and $\mu$ is absolutely decaying on $[0,1)^d$. Their proof was carried out by playing Schmidt's game directly on $K$. A more efficient approach, however, is to employ the hyperplane absolute game, which is a variant of Schmidt's game first introduced in \cite{MR2981929} as a  generalization of McMullen's absolute game \cite{MR2720230}.

 Let $\gamma\in(0,1/3)$ and let $K\subset\mathbb R^d$ be closed. The $\gamma$-hyperplane absolute game on $K$ is defined as follows. Bob starts the game by selecting an arbitrary closed ball $B_1$ with centre in $K$. At each $i$-move, after Bob has chosen $B_i$, Alice chooses a set $A_i=(\mathcal L_i)_{\varepsilon_i}$, where $\mathcal L_i$ is a affine hyperplane of codimension one in $\mathbb R^d$ and $0<\varepsilon_i\le \gamma |B_i|/2$. Bob then replies with a closed ball $B_{i+1}$ centred in $K$ such that
\[
B_{i+1}\subseteq B_i\setminus A_i,\qquad |B_{i+1}|\ge \gamma |B_i|.
\]
A subset $S\subseteq\mathbb R^d$ is called $\gamma$-hyperplane absolute winning on $K$ if Alice possesses a strategy forcing $\bigcap_{i=1}^\infty B_i\cap S\ne\varnothing$, no matter how Bob chooses his balls. The set $S$ is said to be hyperplane absolute winning on $K$ if there is $\gamma_0>0$ such that $S$ is $\gamma$-hyperplane absolute winning on $K$ for every $0<\gamma\le\gamma_0$. In dimension $d=1$, this notion reduces to the absolute winning property introduced by McMullen \cite{MR2720230}. 

 Before stating the properties of the hyperplane absolute game, we need to introduce an auxiliary notion: hyperplane diffuse sets. A set $E\subseteq\mathbb R^d$ is called $\gamma$-hyperplane diffuse if there exists $r_0>0$ with the following property: for every $\mathbf x\in E$ and every affine hyperplane $\mathcal L$ of codimension one,
\[
E\cap \bigl(B(\mathbf x,r)\setminus(\mathcal L)_{\gamma r}\bigr)\ \neq\ \varnothing,\qquad\forall \ 0<r\leq r_0.
\]
   We simply say that $E$ is hyperplane diffuse if it is $\gamma$-hyperplane diffuse for some $\gamma>0$. For example, $[0,1]^d$ is hyperplane diffuse whereas any lower dimensional submanifold of $\R^d$ is not. In general, the support of an absolutely decaying measure is hyperplane diffuse (see the Proposition 5.1 of \cite{MR2981929} for a proof). We are now ready to collect some fundamental properties of hyperplane absolute winning sets and hyperplane diffuse sets.
  \begin{proposition}[{\cite{MR2981929}}]\label{proofhypdi}
      Let $K\subseteq L$ be two closed subsets in $\R^d$ and let $E\subseteq\R^d$. We obtain that
      \begin{itemize}
          \item[(i)] if $K$ is hyperplane diffuse and $E$ is hyperplane absolute winning on $K$, then $E$ is winning on $K$.
          \medskip

          \item[(ii)] if both $K$ and $L$ are hyperplane diffuse and $E$ is hyperplane absolute winning on $L$, then $E$ is hyperplane absolute winning on $K$.
          \medskip

          \item[(iii)] if there exists $\gamma_0$ such that for any $n\in\N$, $E_n$ is $\gamma$-hyperplane absolute winning on $K$ for any $0<\gamma\leq\gamma_0$, then $\bigcap_{n=1}^{\infty}E_n$ is $\gamma$-hyperplane absolute winning on $K$ for any $0<\gamma\leq\gamma_0$.
      \end{itemize}
  \end{proposition}
 \noindent The Proposition \ref{proofhypdi} (i) and (ii) illustrate the power of the hyperplane absolute game: the property of being hyperplane absolute winning on the full space transfers automatically to a wide class of hyperplane diffuse subsets. By combining this with Proposition \ref{dimhlowdima}, we have the following direct consequence. 
 \begin{corollary}\label{Corhypdif}
     If $E$ is hyperplane absolute winning on $[0,1]^d$ and $K$ is a hyperplane diffuse closed set in $[0,1]^d$, then $\dimH(E\cap K)\geq\dimL(K)$.
 \end{corollary} 

 There are two usual ways to improve the above inequality replacing $\dimL$ by $\dimH$ : (i) prove that $\dimL(K)=\dimH(K)$; (ii) find a sequence $\{K_n\}_{n\in\N}$ of hyperplane diffuse subsets of $K$ such that $\dimL(K_n)=\dimH(K_n)\to\dimH(K)$ as $n\to\infty$. We intend to present some applications of these two principles mainly for the case where $K$ belongs to some classical class of fractals, such as  self-conformal sets and self-affine sets.
 
  Let $K$ be a self-conformal set in $\R^d$ generated by $C^{1+\kappa}$ conformal IFS that satisfies the open set condition (OSC).  It is well known that the restricted Hausdorff measure $\mu:=\cH^{s}|_{K}$, where $s:=\dimH(K)$, is $s$-Ahlfors regular. This implies that $\dimL(K)=\dimH(K)$. Moreover, if in addition that $K$ is irreducible, that is $K$ is not contained in a $C^1$ submanifold of $\R^d$ with codimension one, then the Theorem 1.6 of \cite{MR4283274} implies that the measure $\mu$ is also absolutely decaying, and hence $K$ is hyperplane diffuse. By combining these information with Corollary \ref{Corhypdif}, we have $\dimH(E\cap K)=\dimH(K)$ if $E$ is hyperplane absolute winning on $[0,1]^d$ and $K$ is a self-conformal set in $[0,1]^d$ that satisfies the above assumptions. After that, it is natural to ask whether the same equality is valid if the open set condition is dropped. The affirmative answer is obtained in a recent paper of Huang, Li, Wang and Yuan \cite{huang2025win}, where the key step is to approximate, in the sense of Hausdorff dimension, an irreducible self-conformal set by a sequence of hyperplane diffuse self-conformal sets, each satisfying the strong separation condition.

  Compared with self-conformal sets, little was known about self-affine sets, except for Bara\'{n}ski sponge. We first introduce some usual notions. A self-affine set $K\subseteq[0,1]^d$ generated by a diagonal self-affine IFS
  $$\Phi=\{\varphi_i(\bfx)=(\varphi_{i,1}(x_1),\varphi_{i,2}(x_2),...,\varphi_{i,d}(x_d))\}_{i\in\cA}$$
  is called a self-affine sponge. It is said to be a Barański sponge if for each $j=1,2,...,d$, the IFS $\{\varphi_{i,j}\}_{i\in\cA}$ on $\R$ satisfies the open set condition. If, moreover, all contractions share the same coordinatewise contraction rates, it is called a Sierpiński sponge. We say that $\Phi$ has distinguishable coordinates if for any distinct $j,k\in\{1,2,...,d\}$ there exists $i\in\cA$ such that $|\varphi_{i,j}'|\neq|\varphi_{i,k}'|$. For a self-affine IFS $\Phi=\{\varphi_i\}_{i\in\cA}$ and a probability vector $\bfp=(p_i)_{i\in\cA}$, we denote $\mu_{\bfp}$ as the unique probability measure that satisfies
  \begin{equation*}
      \mu_{\bfp}\ = \ \sum_{i\in\cA}p_i\cdot\mu_{\bfp}\circ\varphi_i^{-1}.
  \end{equation*}
  The dynamical dimension of the IFS $\Phi$ is then defined as
  \begin{equation*}
      \dimD(\Phi)\ := \ \sup\{\dim(\mu_{\bfp}):\bfp\ \text{is a probability vector}\},
  \end{equation*}
 where $\dim(\mu)$ is the infimum of Hausdorff dimensions of Borel sets that have full $\mu$-measures. By definition, the dynamical dimension of $\Phi$ is always a lower bound of the Hausdorff dimension of $K$.

 In \cite{MR3904182}, the authors  introduced the concept of irreducibility for diagonal self-affine IFS (distinct from the corresponding notion in the self-conformal setting discussed earlier)  and proved that for any Bara\'{n}ski sponge  $K\subseteq[0,1]^d$ whose associated IFS $\Phi$ is irreducible with distinguishable coordinates and for any hyperplane absolute winning set $E\subseteq[0,1]^d$, one has $\dimH(E\cap K)\geq\dimD(\Phi)$.  We note that although  their main theorem was stated only for intersections with
 \begin{equation*}
   \BAD_d:=\left\{\bfx\in[0,1]^d:\exists\, c(\bfx)>0\text{\ such that}\ \left|\bfx-\frac{\bfp}{q}\right|>\frac{c(\bfx)}{q^{1+1/d}}\,,\ \forall\ (\bfp,q)\in\Z^d\times\N\right\},
\end{equation*}
their arguments actually established results concerning intersections with hyperplane absolute winning sets.
 They achieved this by finding, for any probability vector $\bfp$, a sequence $\{K_n\}_{n\in\N}$ of self-affine subsets of $K$ which are hyperplane diffuse such that
 \begin{equation*}
     \dimL(K_n) \to \dim(\mu_{\bfp})\qquad(n\to\infty).
 \end{equation*}
 In the particular case where $K$ is a Siepi\'{n}ski sponge whose underlying IFS $\Phi$ is irreducible with distinguishable coordinates, the Hausdorff dimension of $K$ is equal to the dynamical dimension of $\Phi$ \cite{MR1389626}, and hence $K\cap E$ has full Hausdorff dimension in $E$.  Recently, Anttila and Fraser \cite{anttila2026} obtained a similar results for intersections with a class of non-linear and non-conformal carpets. To date, the following problem posed in \cite{MR3904182} remains open.
 \begin{problem}
     Let $K\subseteq[0,1]^d$ be a Bara\'{n}ski sponge whose associated IFS $\Phi$ is irreducible with distinguishable coordinates. Suppose that $\dimD(\Phi)<\dimH(K)$. Is it still true that $\dimH(E\cap K)=\dimH(K)$ for any hyperplane absolute winning set $E\subseteq[0,1]^d$?
 \end{problem}

The power of the hyperplane absolute winning property in intersection theory has been illustrated as above. To conclude this subsection, we give several examples of systems for which $\tN(T,\cG)$ is hyperplane absolute winning. Indeed, the sets $\tN_1(T,\cG)$ from \cite{huang2025win,lambert2025}, also described in the last paragraph of Section~\ref{PEM}, are hyperplane absolute winning. Furthermore, Huang, Li, Wang and Yuan \cite{huang2025win} also studied a class of product systems in high dimensions. More precisely, let $T_1:[0,1)\to[0,1)$ be a piecewise $C^{1+\kappa}$ expanding map with finite branches or infinitely full branches, and let $T_2:[0,1)^{d-1}\to[0,1)^{d-1}$ be arbitrary. Define $T:[0,1)^d\to[0,1)^d$ by
\[
T(\mathbf x):=(T_1(x_1),\,T_2(\mathbf x_2)),\qquad \mathbf x=(x_1,\mathbf x_2)\in[0,1)\times[0,1)^{d-1}.
\]
They proved that $\tN_d(T,\cG)$ is hyperplane absolute winning whenever $\cG=\{g_n\}_{n\in\mathbb N}$ is a sequence of uniformly Lipschitz functions on $[0,1)^d$. This result is applicable to $d$-dimensional matrix transformations $T_A:[0,1)^d\to[0,1)^d$ whose underlying matrices are of the form
\begin{equation*}
    A=\left(\begin{aligned}
        &\beta & \\
        & & B
    \end{aligned}\right)\qquad\text{where $|\beta|>1$ and $B$ is any $(d-1)\times (d-1)$ real matrix}.
\end{equation*}
However, the following  problem for general real matrix transformations remains unsolved.

\begin{problem}
    Let $A$ be a $d\times d$ real matrix with at least one eigenvalue whose absolute value is greater than $1$. Let $\cG=\{g_n\}_{n\in\N}$ be a sequence of uniformly Lipschitz functions on $[0,1)^d$. Is it true that $\dimH(\tN_d(T_A,\cG))=d$? Is $\tN_d(T_A,\cG)$ hyperplane absolute winning on $[0,1)^d$?
\end{problem}
 
\subsection*{Acknowledgements} 
We would like to thank Wei-Wen Kuang and Bo Wang for their useful comments on this survey.
BL was supported by National Key R\&D Program of China (No. 2024YFA1013700), NSFC 12271176, 12671110, 12571205, Guangdong Natural Science Foundation 2024A1515010946 and Guangdong Major Project of Basic Research 2026B0303000003.
ZH was supported by NSFC 12501115, Beijing Natural Science Foundation 1254048 and the Science Foundation of China University of Petroleum, Beijing (2462023SZBH013).

\bibliographystyle{plain}
	\bibliography{ref.bib}

\end{document}